\documentclass[a4paper,11pt]{amsart}
\usepackage{amsmath, amsfonts, amssymb, amsthm, amscd}
\usepackage{graphicx}
\usepackage{perpage} 
\usepackage{url}
\usepackage{color}

\usepackage[T1]{fontenc}

\usepackage[a4paper,scale={0.72,0.74},marginratio={1:1},footskip=7mm,headsep=10mm]{geometry}

\usepackage[pdftex]{hyperref}

\numberwithin{equation}{section}

\newtheorem{theorem}{Theorem}[section]
\newtheorem{lemma}[theorem]{Lemma}
\newtheorem{proposition}[theorem]{Proposition}
\newtheorem{corollary}[theorem]{Corollary}
\newtheorem{remark}[theorem]{Remark}

\newcommand{\bbE}{{\ensuremath{\mathbb E}} }

\newcommand{\bbP}{{\ensuremath{\mathbb P}} }

\newcommand{\bbZ}{{\ensuremath{\mathbb Z}} }

\newcommand{\cA}{{\ensuremath{\mathcal A}} }
\newcommand{\cB}{{\ensuremath{\mathcal B}} }
\newcommand{\cC}{{\ensuremath{\mathcal C}} }
\newcommand{\cD}{{\ensuremath{\mathcal D}} }

\newcommand{\cF}{{\ensuremath{\mathcal F}} }
\newcommand{\cG}{{\ensuremath{\mathcal G}} }
\newcommand{\cH}{{\ensuremath{\mathcal H}} }
\newcommand{\cI}{{\ensuremath{\mathcal I}} }

\newcommand{\cK}{{\ensuremath{\mathcal K}} }
\newcommand{\cL}{{\ensuremath{\mathcal L}} }

\newcommand{\cN}{{\ensuremath{\mathcal N}} }

\newcommand{\cS}{{\ensuremath{\mathcal S}} }

\newcommand{\cV}{{\ensuremath{\mathcal V}} }
\newcommand{\cW}{{\ensuremath{\mathcal W}} }

\newcommand{\ga}{\alpha}

\newcommand{\gd}{\delta}

\newcommand{\gep}{\varepsilon}       

\newcommand{\gs}{\sigma}

\renewcommand{\tilde}{\widetilde}          
\DeclareMathSymbol{\leqslant}{\mathalpha}{AMSa}{"36} 
\DeclareMathSymbol{\geqslant}{\mathalpha}{AMSa}{"3E} 
\DeclareMathSymbol{\eset}{\mathalpha}{AMSb}{"3F}     
\newcommand{\dd}{\text{\rm d}}             

\DeclareMathOperator*{\union}{\bigcup}       
\newcommand{\suptwo}[2]{\sup_{\substack{#1 \\ #2}}} 
\newcommand{\sumtwo}[2]{\sum_{\substack{#1 \\ #2}}} 

\newcommand{\R}{\mathbb{R}}

\newcommand{\Z}{\mathbb{Z}}
\newcommand{\N}{\mathbb{N}}

\def\bs{\boldsymbol}

\newcommand{\PEfont}{\mathrm}

\newcommand{\p}{\ensuremath{\PEfont P}}
\newcommand{\e}{\ensuremath{\PEfont E}}
\DeclareMathOperator{\sign}{sign}

\newcommand\bP{\ensuremath{\bs{\mathrm{P}}}}
\newcommand\bE{\ensuremath{\bs{\mathrm{E}}}}

\newcommand{\ind}{{\sf 1}}
\newcommand{\Po}{{\rm Po}}

\renewcommand{\epsilon}{\varepsilon}
\renewcommand{\theta}{\vartheta}
\renewcommand{\phi}{\varphi}
\DeclareMathOperator\argmax{arg\, max}

\def\rw{\mathrm{w}}
\def\rz{\mathrm{z}}

\newenvironment{myenumerate}{%
\renewcommand{\theenumi}{\arabic{enumi}}%
\renewcommand{\labelenumi}{{\rm(\theenumi)}}%
\begin{list}{\labelenumi}
        {%
        \setlength{\itemsep}{0.4em}%
        \setlength{\topsep}{0.5em}%
        \setlength\leftmargin{2.45em}%
        \setlength\labelwidth{2.05em}%
        \setlength{\labelsep}{0.4em}%
        \usecounter{enumi}%
        }%
        }%
{\end{list}
}

\renewenvironment{enumerate}{
\begin{myenumerate}}%
{\end{myenumerate}}

\newenvironment{myitemize}{%
\begin{list}{$\bullet$}%
        {%
        \setlength{\itemsep}{0.4em}%
        \setlength{\topsep}{0.5em}%
        \setlength\leftmargin{2.45em}%
        \setlength\labelwidth{2.05em}%
        \setlength{\labelsep}{0.4em}%
        }%
        }%
{\end{list}}

\renewenvironment{itemize}{
\begin{myitemize}}%
{\end{myitemize}}

\MakePerPage[2]{footnote} 

\author{Francesco Caravenna}

\address{Dipartimento di Matematica e Applicazioni, Universit\`a
degli Studi di Milano-Bicocca, via Cozzi 53, 20125 Milano, Italy}

\email{francesco.caravenna@unimib.it}

\author{Philippe Carmona}

\address{Laboratoire de Math\'ematiques Jean Leray UMR 6629,
Universit\'e de Nantes, 2 Rue de la Houssini\`ere,
BP 92208, F-44322 Nantes Cedex 03, France}

\email{carmona@math.univ-nantes.fr}

\author{Nicolas P\'etr\'elis}

\address{Laboratoire de Math\'ematiques Jean Leray UMR 6629,
Universit\'e de Nantes, 2 Rue de la Houssini\`ere,
BP 92208, F-44322 Nantes Cedex 03, France}

\email{nicolas.petrelis@univ-nantes.fr}

\thanks{F.C. gratefully acknowledges the support of
the University of Padova under grant CPDA082105/08.}



\date{\today}

\title[Discrete-time PAM with heavy tailed potential]
{The discrete-time parabolic Anderson model\\ with heavy-tailed potential}

\date{\today}

\begin{document}

\begin{abstract}
We consider a discrete-time version of the parabolic Anderson model.
This may be described as a model for a directed $(1+d)$-dimensional polymer interacting with a random potential, 
which is constant in the deterministic direction and i.i.d. in the $d$ orthogonal directions.
The potential at each site is a
positive random variable with a polynomial tail at infinity. We show that,
as the size of the system diverges,
the polymer extremity is localized
almost surely at one single point which grows ballistically.
We give an explicit characterization of the localization point and of
the typical paths of the model.
\end{abstract}

\maketitle


\section{Introduction and results}

The model we consider is built on two main ingredients,
a random walk $S$ and a random potential $\xi$.
We start describing these ingredients.
A word about notation: throughout the paper, 
we denote by $|\cdot|$ the $\ell^1$ norm on $\R^d$, 
that is $|x| = |x_1| + \ldots + |x_d|$
for $x = (x_1, \ldots, x_d)$, and we set
$\cB_N := \{x \in \Z^d: |x| \le N\}$.


\subsection{The random walk}

Let $S = \{S_k\}_{k\ge 0}$ denote the coordinate process
on the space $\Omega_S := (\Z^d)^{\N_0 := \{0,1,2,\ldots\}}$,
that we equip as usual with the product topology and $\gs$-field.
We denote by $\p$ the law on $\Omega_S$ under which $S$ is a (lazy)
nearest-neighbor random walk started at zero, that is $\p(S_0 = 0) = 1$
and under $\p$ the variables $\{S_{k+1}-S_k\}_{k\ge 0}$ are i.i.d.
with $\p(S_1 = y) = 0$ if $|y| > 1$. We also assume the
following irreducibility conditions:
\begin{equation}\label{eq:rw0}
	\p(S_1 = 0) \,=:\, \kappa \,>\, 0\, \qquad \text{and} \qquad
	\p(S_1 = y) \,>\, 0 \quad \forall y \in \Z^d \text{ with } |y| = 1 \,.
\end{equation}
The usual assumption $\e(S_1) = 0$ is not necessary.
For $x \in \Z^d$, we denote by $\p_x$ the law of the random walk started at $x$,
that is $\p_x(S \in \cdot) := \p(S + x \in \cdot)$.

We could actually deal with random walks with finite range, i.e.,
for which there exists $R>0$ such that $\p(S_1 = y) = 0$ if $|y| > R$,
but we stick for simplicity to the case $R=1$.



\subsection{The random potential}

We let $\xi = \{\xi(x)\}_{x \in \Z^d}$ denote a family of i.i.d. random
variables taking values in $\R^+$,
defined on some probability space $(\Omega_\xi, \cF, \bbP)$,
which should not be confused with $\Omega_S$. We assume that
the variables $\xi(x)$ are Pareto distributed, that is
\begin{equation} \label{eq:pareto}
        \bbP( \xi(0) \in \dd x ) = \frac{\ga}{x^{1+\ga}}\,
        \ind_{[1,\infty)}(x) \, \dd x \,,
\end{equation}
for some $\alpha \in (0,\infty)$.
Although the precise assumption \eqref{eq:pareto} on the law of $\xi$
could be relaxed to a certain extent, we prefer to keep it for the sake of simplicity.

In the sequel we could work with the product space $\Omega_S \times \Omega_\xi$, equipped with
the product probability $\p \otimes \bbP$,
under which $\xi$ and $S$ are independent, but it is actually not necessary,
because $\xi$ and $S$ will play on a separate level, as it will be clear in a moment.


\subsection{The model}

Given $N \in \N := \{1,2,3,\ldots\}$ and a $\bbP$-typical
realization of the variables $\xi = \{\xi(y)\}_{y \in \Z^d}$,
our model is the probability $\bP_{N,\xi}$ on $\Omega_S$ defined by
\begin{equation} \label{eq:model}
	\frac{\dd \bP_{N,\xi}}{\dd \p}(S) \,:=\,
	\frac{1}{U_{N,\xi}} \, e^{H_{N,\xi}(S)} \,,
	\qquad \text{where} \qquad H_{N,\xi}(S) \,:=\, \sum_{i=1}^N \xi(S_i)\,,
\end{equation}
and the normalizing constant $U_{N,\xi}$ (\emph{partition function}) is of course
\begin{equation} \label{eq:U}
        U_{N,\xi} \,:=\, \e \left[ e^{H_{N,\xi}(S)} \right] \,=\,
        \e \left[ \exp\left(\sum_{i=1}^N \xi(S_i)\right) \right] \,.
\end{equation}
We stress that we are dealing with a \emph{quenched disordered model}:
we are interested in the properties of the law $\bP_{N,\xi}$
for $\bbP$-typical but \emph{fixed} realizations of the potential $\xi$.

Let us also introduce the
\emph{constrained partition function} $u_{N,\xi}(x)$, defined for $x \in \Z^d$ by
\begin{equation} \label{eq:u}
        u_{N,\xi}(x) \,:=\, 
        \e \left[ \exp\left(\sum_{i=1}^N \xi(S_i)\right)
		\ind_{\{S_N = x\}} \right] \,,
\end{equation}
so that $U_{N,\xi} = \sum_{x\in\Z^d} u_{N,\xi}(x)$.
Note that the law of $S_N$ under $\bP_{N,\xi}$ is given by
\begin{equation} \label{eq:p}
	p_{N,\xi}(x) \,:=\, \bP_{N,\xi}(S_N = x) \,=\, \frac{u_{N,\xi}(x)}{U_{N,\xi}}
	\,=\, \frac{u_{N,\xi}(x)}{\sum_{y\in\Z^d} u_{N,\xi}(y)} \,.
\end{equation}

The law $\bP_{N,\xi}$ admits the following interpretation:
the trajectories $\{(i, S_i)\}_{0 \le i \le N}$ model the configurations
of a $(1+d)$-dimensional directed polymer of length $N$ which interacts with the
random potential (or environment) $\{\xi(x)\}_{x\in\Z^d}$.
The random variable $\xi(x)$ should be viewed as a reward sitting at site $x \in \Z^d$,
so that the energy of each polymer configuration is given by the sum of the rewards visited by the polymer.
On an intuitive ground, the polymer configurations
should target the sites where the potential takes very large values.
The corresponding energetic gain entails of course an entropic loss, which however should not be too
relevant, in view of the heavy tail assumption \eqref{eq:pareto}.
As we are going to see, this is indeed
what happens, in a very strong form.

Besides the directed polymer interpretation, 
$\bP_{N,\xi}$ is a law on $\Omega_S = (\Z^d)^{\N_0}$ which
may be viewed as
a natural penalization of the random walk law $\p$. In particular,
when looking at the process $\{S_k\}_{k\ge 0}$ under the law $\bP_{N,\xi}$,
we often consider $k$ as a time parameter.

\begin{remark}\rm
An alternative interpretation of our model is to describe the
spatial distribution of a population evolving in time.
At time zero, the population consists of one
individual located at the site $x=0 \in \Z^d$.
In each time step, every individual in the population performs one step of the random walk $S$,
independently of all other individuals,
jumping from its current site $x$ to a site $y$ (possibly $y=x$)
and then splitting into a number of individuals
(always at site $y$) distributed like a $\Po(e^{\xi(y)})$,
where $\Po(\lambda)$ denotes the Poisson distribution of parameter $\lambda > 0$.
The expected number of individuals at site $x \in \Z^d$ at time $N \in \N$
is then given by $u_{N,\xi}(x)$, as one checks easily.
\end{remark}

\begin{remark}\rm
Our model is somewhat close in spirit to the much studied
\emph{directed polymer in random environment} \cite{MR1939654,MR2073332,MR2579463}, in which the rewards
$\xi(i,x)$ depend also on $i \in \N$ (and are usually chosen to be jointly i.i.d.).
In our model, the rewards are constant in the ``deterministic direction'' $(1,0)$,
a feature which makes the environment much more attractive from a localization viewpoint.
Notice in fact that a site $x$ with a large reward $\xi(x)$ yields
a favorable straight corridor $\{0, \ldots, N\} \times \{x\}$ for the polymer
$\{(i, S_i)\}_{0 \le i \le N}$.

We also point out that the so-called \emph{stretched polymer in random
environment} with a fixed length, considered e.g. in
\cite{cf:IV2}, is a model which provides an interpolation
between the directed polymer in random environment and our model.
\end{remark}


\subsection{The main results}

\smallskip

The closest relative of our model is
obtained considering the continuous-time
analog $\hat u_{t,\xi}(x)$ of \eqref{eq:u}, defined for $t \in [0,\infty)$ and $x \in \Z^d$ by
\begin{equation} \label{eq:hu}
        \hat u_{t,\xi}(x) \,:=\, 
        \e \left[ \exp\left(\int_0^t \xi(\hat S_u) \, \dd u\right)
		\ind_{\{\hat S_t = x\}} \right] \,,
\end{equation}
where $(\{\hat S_u\}_{u\in [0,\infty)}, \p)$ denotes the
continuous-time, simple symmetric random walk on $\Z^d$.
One can check that the function $\hat u_{t,\xi}(x)$ is the solution
of the following Cauchy problem:
\begin{equation*}
\left\{
\begin{split}
	\tfrac{\partial}{\partial t} \hat u_{t,\xi}(x) & \,=\, \Delta \hat u_{t,\xi}(x)
	+ \xi(x) \, \hat u_{t,\xi}(x) \\
	\hat u_{0,\xi}(x) & \,=\, \ind_{0}(x)
\end{split}
\right. \qquad \text{for } (t,x) \in (0,\infty) \times \Z^d \,,
\end{equation*}
known in the literature as the \emph{parabolic Anderson problem}.
We refer to \cite{cf:GM,cf:GK,cf:HKM} and references
therein for the physical motivations behind this
problem and for a survey of the main results.

When the potential $\xi$ is i.i.d. with heavy tails like in \eqref{eq:pareto}
and $\alpha > d$,
the asymptotic properties as $t\to\infty$ of the function $\hat u_{t,\xi}(\cdot)$
were investigated in \cite{cf:KLMS}, showing that a very strong form of localization
takes place: for large $t$, the function $\hat u_{t,\xi}(\cdot)$ is essentially concentrated
at two points almost surely and at a single point in probability.
More precisely, for all $t > 0$ and $\xi \in \Omega_\xi$
there exist $\hat\rz_{t, \xi}^{(1)}, \hat\rz_{t, \xi}^{(2)} \in \Z^d$ such that
\begin{gather} \label{eq:KLMSas}
	\lim_{t\to\infty} 
	\frac{\hat u_{t,\xi}(\hat\rz_{t,\xi}^{(1)}) + \hat u_{t,\xi}(\hat\rz_{t,\xi}^{(2)})}
	{\sum_{x\in\Z^d} \hat u_{t,\xi}(x)} \,=\, 1 \,, \qquad \bbP\text{-almost surely}\,,\\
	\label{eq:KLMSproba}
	\lim_{t\to\infty} \frac{\hat u_{t,\xi}(\hat\rz_{t,\xi}^{(1)})}
	{\sum_{x\in\Z^d} \hat u_{t,\xi}(x)} \,=\, 1 \,, \qquad \text{in $\bbP$-probability} \,,
\end{gather}
cf. \cite[Theorems~1.1 and~1.2]{cf:KLMS}. The points $\hat\rz_{t, \xi}^{(1)}, 
\hat\rz_{t, \xi}^{(2)}$
are typically at superballistic distance $(t/\log t)^{1+q}$
with $q = d/(\alpha - d) > 0$, cf. \cite[Remark~6]{cf:KLMS}.
We point out that localization at one
point like in \eqref{eq:KLMSproba} cannot hold $\bbP$-almost surely, that is,
the contribution of $\hat\rz_{t,\xi}^{(2)}$ cannot be removed from \eqref{eq:KLMSas}:
this is due to the fact that $\hat u_{t,\xi}(x)$ is a continuous function of $t$
for every fixed $x\in\Z^d$,
as explained in \cite[Remark~1]{cf:KLMS}.

\smallskip

It is natural to ask if the discrete-time counterpart
of $\hat u_{t,\xi}(\cdot)$, i.e., the constrained partition
function $u_{N,\xi}(\cdot)$ defined in \eqref{eq:u}, exhibits similar features.
Generally speaking, models built over discrete-time or continuous-time simple random walks
are not expected to be very different.
However, due to the heavy tail of the potential distribution,
the localization points $\hat\rz_{t, \xi}^{(1)}, \hat\rz_{t, \xi}^{(2)}$ 
of the continuous-time model
grow at a superballistic speed, a feature that is clearly impossible for the discrete-time
model, for which $u_{N,\xi}(x) \equiv 0$ for $|x| > N$.
Another interesting
question is whether for the discrete model one may have
localization at one single point $\bbP$-almost surely.
Before answering, we need to set up some notation.

We recall that $\cB_N := \{x \in \Z^d: |x| \le N\}$.
It is not difficult to check
that the values $\{p_{N, \xi}(x)\}_{x \in \cB_N}$ are all distinct,
for $\bbP$-a.e. $\xi \in \Omega_\xi$ and for all $N \in \N$,
because the potential distribution is continuous, cf. \eqref{eq:pareto}.
Therefore we can set
\begin{equation} \label{eq:rw}
	\rw_{N,\xi} \,:=\, \argmax \big\{ p_{N, \xi}(x):\
	x \in \cB_N \big\} \,,
\end{equation}
and $\bbP$-almost surely $\rw_{N,\xi}$ is a single point in $\Z^d$:
it is the point at which $p_{N, \xi}(\cdot)$ attains its maximum.
We can now state our first main result.

\begin{theorem}[One-site localization]\label{th:main}
We have
\begin{equation} \label{eq:1point}
	\lim_{N\to\infty} p_{N,\xi}(\rw_{N,\xi}) \,=\,
	\lim_{N\to\infty} \frac{u_{N,\xi}(\rw_{N,\xi})}{\sum_{x\in\Z^d} u_{N,\xi}(x)}
	\,=\, 1 \,, \qquad \text{$\bbP(\dd\xi)$-almost surely} \,.
\end{equation}
Furthermore, as $N \to \infty$ we have the following convergence in distribution:
\begin{equation} \label{eq:rweak}
	\frac{\rw_{N, \xi}}{N} \, \Longrightarrow \, \rw \,, \qquad \text{where} \qquad
	\bbP( \rw \in \dd x) \,=\, c_\alpha \, (1-|x|)^\alpha \, \ind_{\{|x| \le 1\}} \,
	\, \dd x \,,
\end{equation}
and $c_\alpha := (\int_{|y| \le 1} (1-|y|)^\alpha \dd y)^{-1}$.
\end{theorem}

\noindent
Recalling the definition \eqref{eq:p} of $p_{N,\xi}(x)$, Theorem~\ref{th:main}
shows that $S_N$ under $\bP_{N,\xi}$ is localized
at the ballistic point $\rw_{N, \xi} \approx \rw \cdot N$. 

Next we look more closely at the localization
site $\rw_{N,\xi}$. We introduce two points
$z_{N,\xi}^{(1)}, z_{N,\xi}^{(2)} \in \Z^d$, defined explicitly in terms of the
potential $\xi$, through
\begin{equation} \label{eq:z12}
\begin{split}
	z_{N,\xi}^{(1)} \,& :=\, \argmax \left\{ \left( 1 - \tfrac{|x|}{N+1} \right) \xi(x) \,:\
	x \in \cB_N \right\} \,, \\
	z_{N,\xi}^{(2)} \,& :=\, \argmax \left\{ \left( 1 - \tfrac{|x|}{N+1} \right) \xi(x) \,:\
	x \in \cB_N \setminus \big\{ z_{N,\xi}^{(1)} \big\} \right\} \,.
\end{split}
\end{equation}
Again, the values of $\{( 1 - \tfrac{|x|}{N+1} ) \xi(x)\}_{x \in \cB_N}$ are $\bbP$-a.s. distinct, by the continuity of the
potential distribution, therefore
$z_{N,\xi}^{(1)}$ and $z_{N,\xi}^{(2)}$ are $\bbP$-a.s. single points in $\cB_N$.
We can now give the discrete-time analogues of \eqref{eq:KLMSas} and \eqref{eq:KLMSproba}.

\begin{theorem}[Two-sites localization]\label{th:theowz}
The following relations hold:
\begin{gather} \label{eq:2points}
	\lim_{N\to\infty} \Big( p_{N,\xi} \big( z_{N,\xi}^{(1)} \big)
	+ p_{N,\xi} \big( z_{N,\xi}^{(2)} \big) \Big)
	\,=\, 1 \qquad \text{$\bbP(\dd\xi)$-almost surely} \,, \\
	\label{eq:1pointbis}
	\lim_{N\to\infty} p_{N,\xi} \big( z_{N,\xi}^{(1)} \big)
	\,=\, 1 \qquad \text{in $\bbP(\dd\xi)$-probability}\,.
\end{gather}
\end{theorem}

\smallskip
\noindent
Putting together Theorems~\ref{th:main} and~\ref{th:theowz},
we obtain the following information on~$\rw_{N,\xi}$.

\begin{corollary} \label{th:corcor}
For $\bbP$-a.e. $\xi \in \Omega_\xi$, we have 
$\rw_{N,\xi} \in \{z_{N,\xi}^{(1)}, z_{N,\xi}^{(2)} \}$
for large $N$. Furthermore,
\begin{equation} \label{eq:wz12}
	\lim_{N\to\infty} \bbP\left(\rw_{N, \xi} = z_{N,\xi}^{(1)}\right) \,=\, 1\,.
\end{equation}
\end{corollary}
\begin{remark}\label{re:zn2} \rm 
We stress that the convergence in \eqref{eq:1pointbis}
does not occur $\bbP(\dd\xi)$-almost surely, i.e., $\rw_{N,\xi}$ is not 
equal to $z_N^{(1)}$ for all $N$ large enough,
at least in dimension $d=1$. In fact, in Appendix~\ref{ap:zn2} we show
explicitly that when $d=1$
\begin{equation} \label{eq:wz112}
	\bbP\left(\rw_{N,\xi} = z_{N,\xi}^{(2)} 
	\text{ for infinitely many $N$} \right) \,=\, 1 \,.
\end{equation}
We believe that \eqref{eq:wz112} remains true also for $d > 1$.

\end{remark}

\smallskip

The proof of two sites localization given in \cite{cf:KLMS} for the
continuous-time model is quite technical and
exploits tools from potential theory and spectral analysis.
We point out that 
such tools can be applied also in the discrete-time setting,
but they turn out to be unnecessary. Our proof is indeed based 
on shorter and simpler geometric arguments. 
For instance, we exploit the fact that before reaching a site $x\in \Z^d$
a discrete-time random walk path must visit a least $|x|-1$ different sites ($\neq x$) 
and spend at each of them a least one time unit. 
Of course, this is no longer true for continuous-time random walks.



\subsection{Further path properties}

Theorem \ref{th:main} states that $\bbP(\dd\xi)$-a.s. the probability measure 
$\bP_{N,\xi}$ concentrates on the subset of $\Omega_S$ 
gathering those
random walk trajectories $S$ such that $S_N = \rw_{N,\xi}$. 
It turns out that this subset can be radically narrowed.
In fact, we can introduce a
restricted subset $\cC_{N,\xi} \subseteq \Omega_S$ of random walk trajectories,
defined as follows:
\begin{itemize}
\item the trajectories in $\cC_{N,\xi}$ must reach the site $\rw_{N,\xi}$
for the first time before time $N$, following an injective path, and then must remain at 
$\rw_{N,\xi}$ until time $N$;

\item the length of the injective path until $\rw_{N,\xi}$
differs from $|\rw_{N,\xi}$| --- which is the
minimal one --- at most for a small error term
$h_N := (\log\log N)^{2/\alpha}\, N^{1-1/\alpha}$ if $\alpha > 1$
and $h_N := (\log N)^{1+2/\alpha}$ if $\alpha \le 1$ (note that
in any case $h_N = o(N)$);

\item all the sites $x$ visited by the random walk before reaching $\rw_{N,\xi}$ must 
have an associated field $\xi(x)$ that is strictly smaller than $\xi(\rw_{N,\xi})$.
\end{itemize}
More formally, denoting by $\tau_{x} = \tau_{x}(S) := \inf\{n\geq 0\colon\, S_n=x\}$
the first passage time at $x \in \Z^d$ of a random walk trajectory $S$, we set
\begin{equation}\label{eq:defc}
\begin{split}
	\cC_{N,\xi} \,:=\, \Big\{ S \in \Omega_S \colon & \
	S_i\neq S_j \ \forall i<j\leq  \tau_{\rw_{N,\xi}}\,, \
	S_i = \rw_{N,\xi} \ \forall i \in \{\tau_{\rw_{N,\xi}}, \ldots, N\} \,, \\
	& \ \xi(S_i) < \xi(\rw_{N,\xi}) \ \forall i< \tau_{w_{N,\xi}} \,, \
	\tau_{\rw_{N,\xi}} \le |\rw_{N,\xi}| + h_N \Big\} \,.
\end{split}
\end{equation}
We then have the following result.

\begin{theorem}\label{th:theopath}
For $\bbP$-a.e. $\xi \in \Omega_\xi$, we have
\begin{equation} \label{eq:theopath}
\lim_{N\to \infty} \bP_{N,\xi}(\cC_{N,\xi})=1.
\end{equation}
\end{theorem}

\begin{remark}\rm \label{rem:1d}
It is worth stressing that in dimension $d=1$
the set $\cC_{N,\xi}$ reduces to \emph{a single $N$-steps 
trajectory}. In fact, we have $\cC_{N,\xi} = \cS^{(N, \rw_{N,\xi})}$,
where we denote by $\cS^{(N, x)}$, for
$x \in \cB_N$, the set of trajectories $S \in \Omega_S$ such that
\begin{equation*}
	S_i \,:=\, \begin{cases}
	i \cdot \sign(x) & \text{for } 0 \le i \le |x| \\
	x & \text{for } |x| \le i \le N
\end{cases} \,.
\end{equation*}
As stated in Corollary~\ref{th:corcor},
for large $N$ the site $\rw_{N,\xi}$ is either $z_{N,\xi}^{(1)}$ or $z_{N,\xi}^{(2)}$.
Note that $z_{N,\xi}^{(1)}$ and $z_{N,\xi}^{(2)}$ are easily determined, by \eqref{eq:z12}.
In order to decide whether $\rw_{N,\xi} = z_{N,\xi}^{(1)}$ or $\rw_{N,\xi} = z_{N,\xi}^{(2)}$,
by Theorem~\ref{th:theopath} it is sufficient to compare the explicit contributions
of just two trajectories, i.e.,
$ \bP_{N,\xi}(\cS^{(N, z_{N,\xi}^{(1)})})$ and
$ \bP_{N,\xi}(\cS^{(N, z_{N,\xi}^{(2)})})$.
More precisely, setting $\kappa(i) := \p(S_1 = i)$
for $i \in \{\pm 1, 0\}$ (so that $\kappa = \kappa(0)$, cf. \eqref{eq:rw0})
and
$$
	b_{N,\xi}(x) \,:=\,
	e^{\sum_{i=1}^{|x|-1} \xi(i \sign(x))
	+ (N+1-|x|) \xi(x)} \kappa(\sign(x))^{|x|}
	\kappa(0)^{N-|x|} \,,
$$
we have $w_{N,\xi} = z_{N,\xi}^{(1)}$ if $b_{N,\xi}(z_{N,\xi}^{(1)})
> b_{N,\xi}(z_{N,\xi}^{(2)})$ and $w_{N,\xi} = z_{N,\xi}^{(2)}$ otherwise.
Therefore, in dimension $d=1$, we have a very explicit characterization of
the localization point $\rw_{N,\xi}$.
\end{remark}


\subsection{Organization of the paper}
\label{sec:organization}

The paper is organized as follows.
\begin{itemize}
\item In Section~\ref{sec:fields} we gather some basic estimates on the field,
that will be the main tool of our analysis.

\item In Section~\ref{sec:loc2} we prove Theorem~\ref{th:theowz}.


\item In Section~\ref{sec:loc1} we prove Theorem~\ref{th:main}.

\item In Section~\ref{sec:paths} we prove Theorem~\ref{th:theopath}.

\item Finally, the Appendixes contain the proofs of some technical results.
\end{itemize}
In the sequel, the dependence on $\xi$ of various quantities, like
$H_{N,\xi}$, $\rw_{N,\xi}$, $z_{N,\xi}^{(1)}$, etc.,
will be frequently omitted for short.


\medskip

\section{Asymptotic estimates for the environment}
\label{sec:fields}

This section is devoted to the analysis of the almost sure asymptotic properties
of the random potential $\xi$. With the exception of Proposition~\ref{th:gap13},
which plays a fundamental role in our analysis,
the proof of the results of this section
are obtained with the standard techniques of extreme values theory and are
therefore deferred to the Appendices~\ref{ap:field} and~\ref{ap:modfield}.

Before starting, we set up some notation.
We say that a property of the field $\xi$ depending on $N \in \N$
holds \emph{eventually $\bbP$-a.s.} if for $\bbP$-a.e.
$\xi \in \Omega_\xi$ there exists
$N_0 = N_0(\xi) < \infty$
such that the property holds for all $N\ge N_0$.
We recall that $|\cdot|$ denotes the $\ell^1$ norm on $\R^d$
and $\cB_N = \{x \in \Z^d:\, |x| \le N\}$. With some abuse
of notation, the cardinality
of $\cB_N$ will be still denoted by $|\cB_N|$. Note that
$|\cB_N| = c_d N^d + O(N^{d-1})$ as $N\to\infty$,
where $c_d = \int_{\R^d} \ind_{\{|x| \le 1\}} \dd x = 2^d/d!$.

\subsection{Order statistics for the field}
\label{sec:statfield}

The order statistics of the field $\{\xi(x)\}_{x \in \cB_N}$ is the set
of values attained by the field rearranged in decreasing order, and
is denoted by
\begin{equation}\label{eq:ddefim}
        X_N^{(1)} > X_N^{(2)} > \dots > X_N^{(|\cB_N|)} > 1 \,.
\end{equation}
For simplicity, when $t \in [1, |\cB_N|]$ is not an integer we still set
$X_N^{(t)} := X_N^{(\lfloor t \rfloor)}$.
For later convenience, we denote by $x_N^{(k)}$ the point
in $\cB_N$ at which the value $X_N^{(k)}$ is attained,
that is $X_N^{(k)} = \xi(x_N^{(k)})$.
We are going to
exploit heavily the following almost sure estimates.

\begin{lemma} \label{th:ubX1}
For every $\gep > 0$, eventually $\bbP$-a.s.
\begin{equation} \label{eq:ubX1}
	\frac{N^{d/\alpha}}{(\log \log N)^{1/\alpha + \epsilon}} \le
	X_N^{(1)} \le N^{d/\ga} \, (\log N)^{1/\ga + \gep} \,.
\end{equation}
For every $\theta > 1$ and $\epsilon > 0$,
eventually $\bbP$-a.s.
\begin{equation} \label{eq:asXklogN}
	\frac{N^{d/\alpha}}{(\log N)^{\theta/\alpha + \epsilon}} \le
	X_N^{((\log N)^\theta)} 
	\le \frac{N^{d/\alpha}}{(\log N)^{\theta/\alpha - \epsilon}} \,.
\end{equation}
There exists a constant $A > 0$ such that eventually $\bbP$-a.s.
\begin{equation} \label{eq:asbig}
	\sup_{(\log N) \le k \le |\cB_N|} \big( k^{1/\alpha} X_N^{(k)} \big) 
	\le A \, N^{d/\alpha} \,.
\end{equation}
\end{lemma}

The proof of Lemma~\ref{th:ubX1} is given in Appendix~\ref{sec:ubX1proof}.
For completeness, we point out that
$X_N^{(1)}/(c_d N^{d/\ga})$ converges in distribution
as $N\to\infty$ toward the law $\mu$ on $(0,\infty)$
with $\mu((0,x]) = \exp(-x^{-\ga})$, called
Fr\'echet law of shape parameter $\ga$, as one can easily prove.

Next we give a lower bound on the gaps
$X_N^{(k)} - X_N^{(k+1)}$ for moderate values
of $k$.

\begin{proposition} \label{th:gap15}
For every $\theta>0$ there exists a constant $\gamma > 0$ such that eventually $\bbP$-a.s.
\begin{equation} \label{eq:gap15}
	\inf_{1 \le k \le (\log N)^\theta}
	\left( X_N^{(k)} - X_N^{(k+1)} \right)
	\ge \frac{N^{d/\alpha}}{(\log N)^\gamma} \,.
\end{equation}
\end{proposition}

\noindent
The proof of Proposition~\ref{th:gap15}
is given in Appendix~\ref{sec:gap15proof}.

\smallskip
\subsection{Order statistics for the modified field}

An important role is played by the
modified field $\{\psi_N(x)\}_{x \in \cB_N}$, defined by
\begin{equation} \label{eq:psi}
        \psi_N(x) := \left( 1 - \frac{|x|}{N+1}\right) \xi(x) \,.
\end{equation}
The motivation is the following: for any given point
$x \in \cB_N$, a random walk trajectory $(S_0, S_1, \ldots, S_N)$ that
goes to $x$ in the minimal number of steps
and sticks in $x$ afterwards has an energetic contribution equal to
$\sum_{i=1}^{|x|-1} \xi(S_i) + (N+1) \psi_N(x)$ (recall \eqref{eq:model}).

The order statistics of the modified field $\{\psi_N(x)\}_{x \in \cB_N}$ will be denoted by
\begin{equation*}
        Z_N^{(1)} > Z_N^{(2)} > \dots > Z_N^{|\cB_N|} \,,
\end{equation*}
and we let $z_N^{(k)}$ be the point in $\cB_N$ at which $\psi_N$ attains $Z_N^{(k)}$,
that is $\psi_N(z_N^{(k)})=Z_N^{(k)}$.
A simple but important observation is that
$Z_N^{(k)}$ is increasing in $N$, for every fixed $k\in\N$, since
$\psi_N(x)$ is increasing in $N$ for fixed $x$.
Also note that $Z_N^{(k)} \le X_N^{(k)}$, because $\psi_N(x) \le \xi(x)$.

\smallskip

%

Our attention will be mainly devoted to $Z_N^{(1)}$ and $Z_N^{(2)}$,
whose almost sure asymptotic behaviors are analogous to that of $X_N^{(1)}$,
cf. \eqref{eq:ubX1}.

\begin{lemma}\label{cori}
For every $\gep > 0$, eventually $\bbP$-a.s.
\begin{equation} \label{eq:lbZ1}
	\frac{N^{d/\alpha}}{(\log \log N)^{1/\alpha + \epsilon}} \le
	Z_N^{(2)} \le Z_N^{(1)} \le N^{d/\ga} \, (\log N)^{1/\ga + \gep} \,.
\end{equation}
\end{lemma}

The proof is given in Appendix~\ref{sec:cori}.
Note that only the first inequality needs to be proved,
thanks to \eqref{eq:ubX1} and to the fact that,
plainly, $Z_N^{(2)} \le Z_N^{(1)} \le X_N^{(1)}$.
 
\smallskip

Next we focus on the gaps between $Z_N^{(1)}, Z_N^{(2)}$ and $Z_N^{(3)}$.
The main technical tool is given by the following easy estimates,
proved in Appendix~\ref{sec:lemZproof}.

\begin{lemma} \label{th:lemZ}
There is a constant $c$ such that for all $N \in \N$ and $\gd \in (0,1)$
\begin{gather} \label{eq:gap123}
        \bbP (Z_N^{(2)} > (1-\gd)Z_N^{(1)}) \le c \,\gd \,,
        \qquad
        \bbP (Z_N^{(3)} > (1-\gd)Z_N^{(1)}) \le c \,\gd^2 \,.
\end{gather}
\end{lemma}

\noindent
As a consequence, we have the following result, which will be crucial in the sequel.

\begin{proposition} \label{th:gap13}
For every $d$ and $\alpha$, there exists $\beta \in (1,\infty)$
such that
\begin{equation} \label{eq:gap13}
        Z_N^{(1)} - Z_N^{(3)} \ge \frac{N^{d/\ga}}{(\log N)^\beta} \,,
	\qquad \text{eventually $\bbP$-a.s.}\,.
\end{equation}
\end{proposition}

Although we do not use this fact explicitly,
it is worth stressing that the gap $Z_N^{(1)} - Z_N^{(2)}$
can be as small as $N^{d/\alpha - 1}$ (up to logarithmic corrections),
hence much smaller than the right
hand side of \eqref{eq:gap13}, cf. Appendix~\ref{sec:gap12?}.
This is the reason behind the fact that localization at the two points
$\{z_{N,\xi}^{(1)}, z_{N,\xi}^{(2)}\}$ can be proved quite directly,
cf. Section~\ref{sec:loc2},
whereas localization at a single point 
$\rw_{N,\xi} \in \{z_{N,\xi}^{(1)}, z_{N,\xi}^{(2)}\}$
is harder to obtain, cf. Section~\ref{sec:loc1}.
Furthermore, one may have $\rw_{N,\xi} \ne z_{N,\xi}^{(1)}$
precisely when the gap $Z_N^{(1)} - Z_N^{(2)}$ is small, cf.
Appendix~\ref{ap:zn2}.

\smallskip

\begin{proof}[Proof of Proposition~\ref{th:gap13}]
For $r \in (0,1)$ (that will be fixed later),
we set $N_k := \lfloor e^{k^r} \rfloor$, for $k\in\N$.
By the second relation in \eqref{eq:gap123}, for $\gamma > 0$
(to be fixed later) we have
\begin{equation*}
        \sum_{k\in\N} \bbP \left( Z_{N_k}^{(1)} - Z_{N_k}^{(3)} \le
        \, \frac{1}{(\log N_k)^\gamma} Z_{N_k}^{(1)} \right)
        \le c_1 \, \sum_{k\in\N} \frac{1}{(\log N_k)^{2\gamma}}
        \le (const.) \sum_{k\in\N} \frac{1}{k^{2r\gamma}} < \infty \,,
\end{equation*}
provided $2r\gamma > 1$. Therefore, by the Borel-Cantelli lemma
and \eqref{eq:lbZ1}, eventually (in $k$) $\bbP$-a.s.
\begin{equation} \label{eq:already}
        Z_{N_k}^{(1)} - Z_{N_k}^{(3)} \ge
        \frac{(N_k)^{d/\ga}}{(\log N_k)^{\gamma + 1}} \,.
\end{equation}

Now for a generic $N \in \N$, let $k \in \N$ be
such that $N_{k-1} \le N < N_{k}$. We can write
\begin{equation*}
        Z_{N}^{(1)} - Z_{N}^{(3)} =
        \big( Z_{N}^{(1)} - Z_{N_{k}}^{(1)} \big)
        + \big( Z_{N_{k}}^{(1)} - Z_{N_{k}}^{(3)} \big)
        + \big( Z_{N_{k}}^{(3)} - Z_{N}^{(3)} \big) \,.
\end{equation*}
We already observed that
$Z_N^{(k)}$ is increasing in $N$,
therefore the third term in the right hand side is non-negative
and can be neglected. From \eqref{eq:already} we then get
for large $N$
\begin{equation} \label{eq:wowo}
        Z_{N}^{(1)} - Z_{N}^{(3)} \ge \frac{(N_k)^{d/\ga}}{(\log N_k)^{\gamma + 1}}
        - \big( Z_{N_{k}}^{(1)} - Z_{N}^{(1)} \big)
        \ge \frac{N^{d/\ga}}{2\,(\log N)^{\gamma + 1}}
        - \big( Z_{N_{k}}^{(1)} - Z_{N}^{(1)} \big) \,,
\end{equation}
because $N_k \ge N$ and $N_k \le 2N$ for large $N$
(note that $N_k/N_{k-1} \to 1$ as $k\to\infty$).

It remains to estimate $Z_{N_{k}}^{(1)} - Z_{N}^{(1)}$. Observe that
$Z_{n}^{(1)} = \psi_{n}(z_{n}^{(1)}) \ge \psi_{n}(z_{n+1}^{(1)})$,
because $Z_{n}^{(1)}$ is the maximum of $\psi_{n}$. Therefore
we obtain the estimate
\begin{equation*}
\begin{split}
        Z_{n+1}^{(1)} - Z_{n}^{(1)} & = \psi_{n+1}(z_{n+1}^{(1)})
        - \psi_{n}(z_{n}^{(1)}) \le \psi_{n+1}(z_{n+1}^{(1)})
        - \psi_{n}(z_{n+1}^{(1)}) \\
        & = \frac{|z_{n+1}^{(1)}| \,
        \xi(z_{n+1}^{(1)})}{(n+1)(n+2)} \le
        \frac{\xi(z_{n+1}^{(1)})}{n} \,,
\end{split}
\end{equation*}
which yields
\begin{equation} \label{eq:wowoZ}
        Z_{N_{k}}^{(1)} - Z_{N}^{(1)} =
        \sum_{n=N}^{N_k-1} \big( Z_{n+1}^{(1)} - Z_{n}^{(1)} \big)
        \le \frac{N_k-N_{k-1}}{N_{k-1}}\, \xi(z_{N_k}^{(1)})
        \le \frac{N_k-N_{k-1}}{N_{k-1}}\, X_{N_k}^{(1)} \,.
\end{equation}
Observe that as $k \to \infty$
\begin{equation}\label{eq:nkineg}
        \frac{e^{k^r} - e^{(k-1)^r}}{e^{(k-1)^r}} =
        e^{k^r - (k-1)^r} - 1 = \frac{r}{k^{1-r}} (1+o(1)) \,.
\end{equation}
Since $N \le N_k=\lfloor e^{k^r}\rfloor$, it comes that $k \ge (\log N)^{1/r}$ and therefore \eqref{eq:nkineg} allows to write
for large $N$
\begin{equation*}
        \frac{N_k-N_{k-1}}{N_{k-1}} \le \frac{1}{(\log N)^{1/r - 1}} \,.
\end{equation*}
Looking back at \eqref{eq:wowo} and \eqref{eq:wowoZ},
by \eqref{eq:ubX1} we then have eventually $\bbP$-a.s.
\begin{equation} \label{eq:wowo2}
        Z_{N}^{(1)} - Z_{N}^{(3)} \ge \frac{N^{d/\ga}}{2\,(\log N)^{\gamma + 1}}
        - \frac{N^{d/\ga}}{(\log N)^{1/r - 1/\ga - 2}} \,.
\end{equation}

The second term in the right hand side of
\eqref{eq:wowo2} can be neglected provided
the parameters $r \in (0,1)$ and $\gamma \in (0,\infty)$
fulfill the condition $1/r - 1/\ga - 2 > \gamma + 1$. We recall
that we also have to obey the condition $2r\gamma > 1$.
Therefore, for a fixed value of $r$, the set
of allowed values for $\gamma$ is the interval
$(\frac{1}{2r}, \frac{1}{r} - \frac{1}{\ga} - 3)$,
which is non-empty if $r$ is small enough.
This shows that the two conditions on $r, \gamma$ can indeed
be satisfied together (a possible choice
is, e.g., $r = \frac{\ga}{6(3\ga + 1)}$ and $\gamma = \frac{4(3\ga +1)}{\ga}$).
Setting $\beta := \gamma +1$, it then follows from \eqref{eq:wowo2}
that equation \eqref{eq:gap13} holds true.
\end{proof}


\medskip
\section{Almost sure localization at two points}
\label{sec:loc2}

In this section we prove Theorem~\ref{th:theowz}.
We first set up some notation and give some preliminary estimates.


%


\subsection{Prelude}

We recall that $z_N^{(1)}$ and $z_N^{(2)}$ are the two
sites in $\cB_N$ at which the modified potential
$\psi_N$, cf. \eqref{eq:psi}, attains its two largest values
$Z_N^{(1)} = \psi_N(z_N^{(1)})$ and $Z_N^{(2)} = \psi_N(z_N^{(2)})$.

It is convenient to define $J_1, J_2 \in \{1, \ldots, |\cB_N|\}$ such that 
\begin{align}\label{eq:defim}
	z_N^{(1)} = x_N^{(J_1)}\,, \qquad 	z_N^{(2)} = x_N^{(J_2)}\,,
\end{align}
where we recall that $x_N^{(k)}$ is the point in $\cB_N$ at which the potential
$\xi$ attains its $k$-th largest value, i.e., $X_N^{(k)} = \xi(x_N^{(k)})$,
cf. Section~\ref{sec:statfield}.
We stress that $J_1$ and $J_2$ are functions of $N$ and
$\xi$, although we do not indicate this explicitly.
An immediate consequence of Lemma~\ref{cori} and relation \eqref{eq:asXklogN}
is the following

\begin{corollary} \label{th:boundJ}
For every $d$, $\alpha$, $\gep>0$, eventually $\bbP$-a.s.
\begin{equation}\label{eq:contr}
        \max\{J_1, J_2\} \leq (\log N)^{1+\gep}. 
\end{equation}
\end{corollary}

Next we define the local time $\ell_N(x)$ of a random walk trajectory $S \in \Omega_S$ by
\begin{equation} \label{eq:lN}
        \ell_N(x) = \ell_N(x,S) = \sum_{i=1}^N \ind_{\{S_i = x\}} \,,
\end{equation}
so that the Hamiltonian $H_N(S)$, cf. \eqref{eq:model}, can be rewritten as
\begin{equation} \label{eq:HlN}
        H_N(S) = \sum_{x \in \cB_N} \ell_N(x) \, \xi(x) \,.
\end{equation}
We also associate to every trajectory $S$ the quantity
\begin{equation} \label{eq:betaN}
		\beta_N(S) := \min\{k\geq 1\colon \ell_N(x_N^{(k)})>0\} \,.
\end{equation}
In words, \emph{$x_N^{(\beta_N(S))}$ is the site in $\cB_N$ which
maximizes the potential $\xi$ among those
visited by the trajectory $S$ before time $N$}. Finally, we introduce the basic events
\begin{equation} \label{eq:AN}
	\cA_{1,N} := \big\{ S \in \Omega_S :\, \beta_N(S) = J_1 \big\}\,,
	\qquad \cA_{2,N} := \big\{ S \in \Omega_S :\, \beta_N(S) = J_2 \big\} \,.
\end{equation}
In words, the event \emph{$\cA_{i,N}$
consists of the random walk trajectories $S$ that before time $N$
visit the site $z_N^{(i)}$ (recall \eqref{eq:defim})
and do not visit any other site $x$ with $\xi(x) > \xi(z_N^{(i)})$}.

\smallskip

It turns out that the localization of $S_N$ at the point $z_N^{(i)}$
is implied by the event $\cA_{i,N}$, i.e.,
for both $i=1,2$ we have
\begin{equation} \label{eq:bigtool}
	\lim_{N \to \infty} \bP_{N,\xi} \big( \cA_{i,N}, \, S_N \ne z_N^{(i)} \big)
	\,=\, 0 \,, \qquad \bbP(\dd\xi)\text{-almost surely}\,.
\end{equation}
The proof is simple. Denoting by $\tilde{\tau}_{N,i}$
the \emph{last} passage time of the random walk in $z_N^{(i)}$
before time $N$, that is
$$
        \tilde{\tau}_{N,i}:=\max\{n\leq N \colon S_n=z_N^{(i)}\} \,,
$$
we can write, recalling \eqref{eq:model},
\begin{align}\label{boundT}
        \bP_{N,\xi} \big( \cA_{i,N}, \, S_N \ne z_N^{(i)} \big) 
        = \sum_{r=0}^{N-1} \frac{\e\big[e^{H_N(S)}\, 
        \ind_{\cA_{i,N}} \, \ind_{\{\tilde\tau_{N,i}=r\}} \big]}{U_N} \,.
\end{align}
We stress that the sum stops at $r=N-1$, because we are on the event
$S_N \ne z_N^{(i)}$. Furthermore, on the event $\cA_{i,N} \cap \{\tilde\tau_{N,i}=r\}$
we have $S_n \not\in \{x_N^{(1)}, \ldots, x_N^{(J_i)}\}$ for all
$n \in \{r+1, \ldots, N\}$ (we recall that $z_N^{(i)} = x_N^{(J_i)}$).
By the Markov property, we can then bound the numerator in
the right hand side of \eqref{boundT} by
\begin{gather}\label{eq:Bnr}
\begin{split}
        \e \Big[e^{H_N(S)}\, \ind_{\cA_{i,N}} \, \ind_{\{\tilde\tau_{N,i}=r\}} \Big]
        &\,\leq\, \e \Big[e^{H_r(S)}\,\ind_{\{S_r=z_N^{(i)}\}}\Big] \, B_{N-r}^{(N,i)} \,, \\
        \text{where} \qquad B_{l}^{(N,i)} & \,:= \,
        \e_{x_N^{(J_i)}}\Big[e^{H_{l}(S)} \, 
        \ind_{\{ S_n \not\in \{x_N^{(1)}, \ldots, x_N^{(J_i)}\} \,
        \forall n = 1, \ldots, l\}}\Big] \, .
\end{split}
\end{gather}
Analogously, for the denominator in the right hand side of \eqref{boundT},
recalling \eqref{eq:rw0}, we have
\begin{equation*}
        U_N \ge \e \big[e^{H_N(S)}\, \ind_{\{S_n = x_N^{(J_i)} , \, \forall n \in
        \{r, \ldots, N\}\}} \big] = \e \Big[e^{H_r(S)}\, \ind_{\{S_r = x_N^{(J_i)}\}}\Big]
        \, \kappa^{N-r} e^{(N-r) X_N^{(J_i)}} \,.
\end{equation*}
Plainly, $B_l^{(N,i)} \le \exp(l X_N^{(J_i + 1)})$, therefore we can write
\begin{align} \label{eq:baust}
\nonumber       
		\bP_{N,\xi} \big( \cA_{i,N}, \, S_N \ne z_N^{(i)} \big) 
		& \le \sum_{r=0}^{N-1} e^{-(N-r) (X_N^{(J_i)} + \log \kappa)} \,
        B_{N-r}^{(N,i)}  = \sum_{l=1}^{N} e^{-l (X_N^{(J_i)} + \log \kappa)} \,
        B_{l}^{(N,i)}\\
        &	\leq \sum_{l=1}^{\infty} 
        e^{-l (X_N^{(J_i)} -X_N^{(J_i+1)}- \log \kappa)} =
        \frac{e^{-(X_N^{(J_i)}-X_N^{(J_i+1)}-\log \kappa)}}
        {1 - e^{-(X_N^{(J_i)}-X_N^{(J_i+1)}-\log \kappa)}} \,.
\end{align}
From Corollary~\ref{th:boundJ} and
Proposition~\ref{th:gap15} it follows that $\bbP(\dd\xi)$-almost surely
$X_N^{(J_i)}-X_N^{(J_i+1)} \to +\infty$ as $N\to\infty$, therefore
\eqref{eq:bigtool} is proved.


\subsection{Proof of \eqref{eq:2points}}

Let us set, for $i=1,2$,
\begin{equation} \label{eq:WN}
\begin{split}
	\cW_{i,N} := & \big\{ S \in \Omega_S :\, \beta_N(S) = J_i,\ S_N=z_N^{(i)} \big\} \\
	= & \big\{S \in \Omega_S:\, S_N = z_N^{(i)},\ \ell_N(x) = 0 \
	\forall x \in \cB_N \text{ such that } \xi(x) > \xi(z_N^{(i)}) \big\} \,.
\end{split}
\end{equation}
In words, the event \emph{$\cW_{i,N}$
consists of those trajectories $S$ such that $S_N = z_N^{(i)}$
and that before time $N$ do not visit any site $x$ with $\xi(x) > \xi(z_N^{(i)})$}.
We are going to prove that
\begin{equation} \label{eq:toto}
	\lim_{N\to\infty} \big( \bP_{N,\xi}(\cW_{1,N}) + \bP_{N,\xi}(\cW_{2,N}) \big)
	= 1 \,, \qquad \bbP(\dd\xi)\text{-almost surely}\,,
\end{equation}
which is a stronger statement than \eqref{eq:2points}. In view of \eqref{eq:bigtool},
it is sufficient to prove that
\begin{equation} \label{eq:tototo}
	\lim_{N\to\infty} \big( \bP_{N,\xi}(\cA_{1,N}) + \bP_{N,\xi}(\cA_{2,N}) \big)
	= 1\,, \qquad \bbP(\dd\xi)\text{-almost surely}\,.
\end{equation}

We start deriving an upper bound on the Hamiltonian $H_N = H_{N,\xi}$
(recall \eqref{eq:model}).
For an arbitrary $k \in \{1, \ldots, |\cB_N|\}$, to be chosen later, recalling
\eqref{eq:lN}, \eqref{eq:HlN}, \eqref{eq:betaN} and the fact that
$\sum_{x\in\Z^d} \ell_N(x) = N$, we can write
\begin{equation} \label{eq:twoH}
\begin{split}
        H_N(S) & = \sum_{i=\beta_N(S)}^k \ell_N(x_N^{(i)}) \xi(x_N^{(i)})
        + \sum_{i=k+1}^{|\cB_N|} \ell_N(x_N^{(i)}) \xi(x_N^{(i)}) \\
        & \le \Bigg( \sum_{i=\beta_N(S)}^k \ell_N(x_N^{(i)}) \Bigg)\, 
        \xi\big(x_N^{(\beta_N(S))}\big) + N\, X_N^{(k+1)} \,.
\end{split}
\end{equation}
Note that $\ell_N\big(x_N^{(\beta_N(S))}\big) > 0$, that is, any trajectory
$S$ visits the site $x_N^{(\beta_N(S))}$ before time $N$, by the very definition
\eqref{eq:betaN} of $\beta_N(S)$. It follows that any trajectory $S$ before time $N$ 
must visit at least $|x_N^{(\beta_N(S))}|$ different
sites, of which at least $|x_N^{(\beta_N(S))}| - k$ are different 
from $x_N^{(1)}$, \ldots, $x_N^{(k)}$. This leads to the basic estimate
\begin{equation}\label{eq:esti}
        \sum_{i=\beta_N(S)}^k \ell_N(x_N^{(i)}) \le N -|x_N^{(\beta_N(S))}| + k \,.
\end{equation}
By \eqref{eq:twoH} and recalling \eqref{eq:psi}, this yields the crucial upper bound
\begin{equation}\label{eq:esti2}
\begin{split}
        H_N(S) & \,\le\, (N+1) \psi_N\big(x_N^{(\beta_N(S))}\big) 
        + (k-1)\, \xi\big(x_N^{(\beta_N(S))}\big) + N\, X_N^{(k+1)} \\
        & \,\le\, (N+1) \psi_N\big(x_N^{(\beta_N(S))}\big) 
        + (k-1)\, X_N^{(1)} + N\, X_N^{(k+1)} \,.
\end{split}
\end{equation}
We stress that this bound holds for all $k \in \{1, \ldots, |\cB_N|\}$
and for all trajectories $S \in \Omega_S$.

Next we give a lower bound on $U_N$ (recall \eqref{eq:U}). We
restrict the expectation to one single $N$-steps random walk trajectory, 
denoted by $S^* = \{S^*_i\}_{0 \le i \le N}$,
that goes to $z_N^{(1)}$ in the minimal number of steps,
i.e. $|z_N^{(1)}|$, and then stays there until epoch $N$.
By \eqref{eq:rw0}, this trajectory
has a probability larger than $e^{-cN}$ for some positive contant $c$,
therefore
\begin{equation} \label{eq:lbU}
        U_N \ge e^{H_N(S^*)-cN} \ge
        e^{\xi(z_N^{(1)}) (N+1 - |z_N^{(1)}|)-cN} = 
        e^{(N+1) \psi_N(z_N^{(1)})-cN} \ge
        e^{(N+1) (Z_N^{(1)} -c)} \,,
\end{equation}
where we have used the definition of $\psi_N$, see \eqref{eq:psi}.

We can finally come to the proof of \eqref{eq:tototo}. For all trajectories
$S \in (\cA_{1,N} \cup \cA_{2,N})^c$ we have $\beta_N(S) \not \in \{J_1, J_2\}$,
therefore $x_N^{(\beta_N(S))} \not \in \{z_N^{(1)}, z_N^{(2)}\}$ and consequently
$\psi_N(x_N^{(\beta_N(S))}) \le Z_N^{(3)}$. From \eqref{eq:esti2} and \eqref{eq:lbU}
we then obtain
\begin{equation}
\begin{split}
	& \bP_{N,\xi} \big( (\cA_{1,N} \cup \cA_{2,N})^c \big) \,=\,
	\frac{\bE( e^{H_N(S)} \, \ind_{(\cA_{1,N} \cup \cA_{2,N})^c})}{U_{N,\xi}} \\
	& \qquad \qquad \,\le\, \exp \left( -(N+1) \left( (Z_N^{(1)} - Z_N^{(3)})
	- X_N^{(k+1)} - \tfrac{k-1}{N+1} X_N^{(1)} - c \right) \right) \,.
\end{split}
\end{equation}
By \eqref{eq:gap13}, there exists $\beta \in (1,\infty)$
such that $Z_N^{(1)} - Z_N^{(3)} \ge N^{d/\alpha}/(\log N)^\beta$
eventually $\bbP$-almost surely.
We now choose $k=k_N = (\log N)^{\theta}$ with $\theta := 3\max\{\beta \alpha, 1\} > 1$.
Applying \eqref{eq:ubX1} with $\gep = 1/\alpha$ and
\eqref{eq:asXklogN} with $\gep = \beta$, we have eventually $\bbP$-a.s.
\begin{equation*}
\begin{split}
		& \left(  (Z_N^{(1)} - Z_N^{(3)})
		- X_N^{(k_N + 1)} - \tfrac{k_N - 1}{N+1} X_N^{(1)} - c \right) \\
        & \qquad \ge N^{d/\ga}
        \bigg( \frac{1}{(\log N)^\beta} - \frac{1}{(\log N)^{2 \beta}} 
        - \frac{(\log N)^{\theta + 2/\alpha}}{N}
        - \frac{c}{N^{d/\ga}} \bigg) = \frac{N^{d/\ga}}{(\log N)^{\beta}} (1+o(1)) \,,
\end{split}
\end{equation*}
therefore, eventually $\bbP$-almost surely,
\begin{equation*}
\begin{split}
	\bP_{N,\xi} (\cA_{1,N}) + \bP_{N,\xi} (\cA_{2,N})
	& \,=\, 1 - \bP_{N,\xi} \big( (\cA_{1,N} \cup \cA_{2,N})^c \big) \\
	& \,\ge\, 1 - \exp \left( -\frac{N^{1 + d/\ga}}{(\log N)^{\beta}} (1+o(1)) \right) \,,
\end{split}
\end{equation*}
which completes the proof of \eqref{eq:tototo}.


\subsection{Proof of \eqref{eq:1pointbis}}

Recalling \eqref{eq:WN}, we are going to prove that
\begin{equation} \label{eq:toto2}
	\lim_{N\to\infty} \bP_{N,\xi}(\cW_{1,N})
	= 1 \,, \qquad \text{in $\bbP(\dd\xi)$-probability}\,,
\end{equation}
which is stronger than \eqref{eq:1pointbis}.
In view of \eqref{eq:bigtool}, it suffices to show that
\begin{equation} \label{eq:tototo2}
	\lim_{N\to\infty} \bP_{N,\xi}(\cA_{1,N})
	= 1 \,, \qquad \text{in $\bbP(\dd\xi)$-probability}\,.
\end{equation}
We actually prove the following: for every
$N \in \N$ there exists a subset $\Gamma_N \subseteq \Omega_\xi$ such that as $N\to\infty$
one has $\bbP(\Gamma_N) \to 1$ and
$\inf_{\xi\in\Gamma_N} \bP_{N,\xi}(\cA_{1,N}) \to 1$, which implies \eqref{eq:tototo2}.

For every trajectory
$S \in (\cA_{1,N})^c$ we have $\beta_N(S) \ne J_1$,
therefore $x_N^{(\beta_N(S))} \ne z_N^{(1)}$ and consequently
$\psi_N(x_N^{(\beta_N(S))}) \le Z_N^{(2)}$. From \eqref{eq:esti2} and \eqref{eq:lbU}
we then obtain
\begin{equation} \label{eq:parapao}
\begin{split}
	\bP_{N,\xi} \big( (\cA_{1,N})^c \big) & \,=\,
	\frac{\bE( e^{H_N(S)} \, \ind_{(\cA_{1,N})^c})}{U_{N,\xi}} \\
	& \,\le\, \exp \left( -(N+1) \left( (Z_N^{(1)} - Z_N^{(2)})
	- X_N^{(k+1)} - \tfrac{k-1}{N+1} X_N^{(1)} - c \right) \right) \,.
\end{split}
\end{equation}
We set $\Gamma_N^{(1)} := \{ Z_N^{(2)} \le (1 - \frac{1}{\log N}) Z_N^{(1)}\}$
and it follows from \eqref{eq:gap123} that $\bbP(\Gamma_N^{(1)}) \to 1$ as $N \to \infty$.
Note that for $\xi \in \Gamma_N^{(1)}$ we have
\begin{equation*}
\begin{split}
	& \left( (Z_N^{(1)} - Z_N^{(2)})
	- X_N^{(k+1)} - \tfrac{k-1}{N+1} X_N^{(1)} - c \right) 
	\ge \left( \tfrac{1}{\log N} Z_N^{(1)} 
	- X_N^{(k+1)} - \tfrac{k-1}{N+1} X_N^{(1)} - c \right) \,.
\end{split}
\end{equation*}
We now fix $k=k_N = (\log N)^{\theta}$ with $\theta := 3\max\{2\alpha, 1\} > 1$.
Applying \eqref{eq:ubX1} with $\gep = 1/\alpha$,
\eqref{eq:asXklogN} with $\gep = 2$ and \eqref{eq:lbZ1}, we have eventually $\bbP$-a.s.
\begin{equation*}
\begin{split}
		& \left(  \tfrac{1}{\log N} Z_N^{(1)} 
		- X_N^{(k_N+1)} - \tfrac{k_N-1}{N+1} X_N^{(1)} - c \right) \\
        & \qquad \ge N^{d/\ga}
        \bigg( \frac{1}{(\log N)^2} - \frac{1}{(\log N)^{4}} 
        - \frac{(\log N)^{\theta + 2/\alpha}}{N}
        - \frac{c}{N^{d/\ga}} \bigg) = \frac{N^{d/\ga}}{(\log N)^{2}} (1+o(1)) \,.
\end{split}
\end{equation*}
In particular, defining
$\Gamma_N^{(2)} := \{ \tfrac{1}{\log N} Z_N^{(1)} - X_N^{(k+1)} - \tfrac{k-1}{N+1} X_N^{(1)} - c
> N^{d/\alpha} / (\log N)^3\}$, we have $\bbP(\Gamma_N^{(2)}) \to 1$
as $N\to\infty$. Setting $\Gamma_N := \Gamma_N^{(1)} \cap \Gamma_N^{(2)}$,
we have $\bbP(\Gamma_N) \to 1$ as $N\to\infty$; furthermore,
by the preceding steps we have that, for all $\xi \in \Gamma_N$,
\begin{equation*}
	\bP_{N,\xi} \big( \cA_{1,N} \big) \,=\,
	1 - \bP_{N,\xi} \big( (\cA_{1,N})^c \big) \,\ge\, 
	1 - \exp \left( (N+1) \frac{N^{d/\alpha}}{(\log N)^3} \right) \,.
\end{equation*}
This completes the proof of \eqref{eq:tototo2}.



\medskip
\section{Almost sure localization at one point}
\label{sec:loc1}

In this section we prove Theorem~\ref{th:main}.
Relation \eqref{eq:1point} is obtained in two steps.
First, we refine the results of the previous section, showing that
\eqref{eq:toto} still holds if we replace the events $\cW_{i,N}$,
$i = 1,2$, that were introduced in \eqref{eq:WN}, by
\begin{equation} \label{eq:tildeWN}
\begin{split}
	\tilde{\cW}_{i,N} \,:=\, & \Big\{ S \in \Omega_S \colon\; 
	\beta_N(S) = J_i, \ S_N=z_N^{(i)}, \ \ell_N(z_N^{(i)})
	> \tfrac{N-|z_N^{(i)}|}{2} \Big\} \\
	\,=\, & \Big\{ S \in \Omega_S \colon\; S_N=z_N^{(i)}, \ \ell_N(z_N^{(i)})
	> \tfrac{N-|z_N^{(i)}|}{2}, \\
	& \qquad \qquad \quad \ell_N(x) = 0 \
	\forall x \in \cB_N \text{ such that } \xi(x) > \xi(z_N^{(i)}) \Big\} \,,	
\end{split}
\end{equation}
that is, if we require that the random walk trajectories spend at $z_N^{(i)}$
at least $(N-|z_N^{(i)}|)/2$ units of time (recall \eqref{eq:lN}). In the second step,
we show that eventually $\bbP(\dd\xi)$-almost surely
\begin{equation} \label{eq:parapeo}
	\max\big\{ \bP_{N,\xi}(\tilde \cW_{1,N}),\, \bP_{N,\xi}(\tilde \cW_{2,N})	\big\}
	\;\gg\; \min \big\{ \bP_{N,\xi}(\tilde \cW_{1,N}),\, \bP_{N,\xi}(\tilde \cW_{2,N}) \big\} \,,
\end{equation}
which yields \eqref{eq:1point}. Finally, we prove \eqref{eq:rweak}
in Section~\ref{sec:rweak}.

%
%


\subsection{Step 1}

In this step we refine \eqref{eq:toto}, showing that
\begin{equation} \label{eq:loctilde}
	\lim_{N\to\infty} \big( \bP_{N,\xi}(\tilde\cW_{1,N}) + 
	\bP_{N,\xi}(\tilde\cW_{2,N}) \big)
	= 1 \,, \qquad \bbP(\dd\xi)\text{-almost surely}\,,	
\end{equation}
where $\tilde\cW_{i,N}$ is defined in \eqref{eq:tildeWN}.
Consider indeed $S\in \cW_{i,N}\setminus\tilde{\cW}_{i,N}$,
with $i\in \{1,2\}$. Before reaching $z_N^{(i)}$, $S$ must visit at least $|z_N^{(i)}|-1$ 
different sites at which, by definition of $\cW_{i,N}$,
the field is smaller than $\xi(z_N^{(i)}) = X_N^{(J_i)}$ (recall \eqref{eq:defim}), hence
\begin{equation*}
	H_N(S) \leq \ell_N(z_N^{(i)}) \, X_N^{(J_i)}
	+ \sum_{j=1}^{|z_N^{(i)}|-1} X_N^{(J_i+j)}
	+ \big( N - \ell_N(z_N^{(i)}) - (|z_N^{(i)}|-1) \big) X_N^{(J_i+1)} \,.
\end{equation*}
Since $\ell_N(z_N^{(i)})\leq (N-|z_N^{(i)}|)/2$ on $\cW_{i,N}\setminus\tilde{\cW}_{i,N}$, we obtain
$$
	H_N(S) \leq \tfrac{N-|z_N^{(i)}|}{2} \, X_N^{(J_i)}
	+ \sum_{j=1}^{|z_N^{(i)}|-1} X_N^{(J_i + j)}
	+ \Big(\tfrac{N-|z_N^{(i)}|}{2} + 1 \Big) X_N^{(J_i + 1)} \,.
$$
Rewriting \eqref{eq:lbU} as 
$U_N\geq e^{(N+1-|z_N^{(i)}|) X_N^{(J_i)}-cN}$ (recall \eqref{eq:psi}),
we can write
\begin{equation}\label{trajneg}
\begin{split}
	\bP_{N,\xi} \big( \cW_{i,N}\setminus\tilde{\cW}_{i,N} \big) & =
	\frac{\e(e^{H_N(S)} \ind_{\{S\in \cW_{i,N}\setminus\tilde{\cW}_{i,N}\}})}{U_{N}} \\
	& \leq e^{c N} \exp\Bigg( -\frac{N-|z_N^{(i)}|}{2} (X_N^{(J_i)}-X_N^{(J_i+1)})
	+ \sum_{j=1}^{|z_N^{(i)}|-1} X_N^{(J_i+j)} \Bigg) \,.
\end{split}
\end{equation}

Applying \eqref{eq:lbZ1} with $\gep=1/\alpha$ and \eqref{eq:ubX1}
with $\gep=\gep/2$, it follows that eventually $\bbP$-a.s. 
\begin{equation*}
Z_N^{(1)}\geq Z_N^{(2)}\geq \tfrac{N^{d/\alpha}}{(\log\log N)^{2/\alpha}}\quad \text{and}\quad  
\max\{X_N^{(J_1)},X_N^{(J_2)}\}\leq N^{d/\alpha}\, (\log N)^{1/\alpha+\gep/2} \,.
\end{equation*}
Since by definition $Z_N^{(i)}= (1-\frac{|z_N^{(i)}|}{N+1}) X_N^{(J_i)}$,
it follows that for both
$i \in \{1,2\}$ and for every $\gep>0$, eventually $\mathbb{P}$-a.s. 
\begin{equation} \label{eq:ecart}
	N-|z_N^{(i)}| \geq \frac{N}{(\log N)^{1/\alpha+\gep}} \,.
\end{equation}
Next we observe that, by the upper bound in \eqref{eq:ubX1} and \eqref{eq:asbig}, we have
\begin{align*}
	\sum_{j=1}^N X_N^{(j)} \le (\log N) X_N^{(1)} +
	\sum_{j= \lceil \log N \rceil}^N X_N^{(j)} 
	\le N^{d/\alpha} \Bigg( (\log N)^{1 + 3/2\alpha} 
	+ \sum_{j= \lceil \log N \rceil}^N 
	\frac{1}{j^{1/\alpha}} \Bigg) \,,
\end{align*} 
therefore there exists a constant $c > 0$
such that, eventually $\bbP$-almost surely,
\begin{equation} \label{eq:sumbou}
	\sum_{j=1}^N X_N^{(j)} \le \begin{cases}
	c \, N^{d/\alpha + 1 - 1/\alpha} & \text{ if } \alpha > 1 \\
	(\log N)^{1 + 3/2\alpha} N^{d/\alpha}  & \text{ if } \alpha \le 1
	\end{cases} \,.
\end{equation}

Looking back at \eqref{trajneg}, we can
apply \eqref{eq:ecart} and \eqref{eq:sumbou} as well as Proposition \ref{th:gap15} 
and Corollary \ref{th:boundJ} to conclude that $\bbP(\dd\xi)$-a.s. 
the right hand side of \eqref{trajneg} vanishes as $N\to \infty$.
Recalling \eqref{eq:toto}, it follows that \eqref{eq:loctilde} holds true,
and the first step is completed.


\subsection{Step 2}

In this step we prove that
\begin{equation}\label{eq:ratio}
	\lim_{N\to \infty} 
	\big|\log \bP_{N,\xi} ( \tilde{\cW}_{1,N} )
	- \log \bP_{N,\xi} (\tilde{\cW}_{2,N}) \big|=\infty \,, \qquad
	\bbP(\dd\xi)\text{-almost surely}\,.
\end{equation}
Together with \eqref{eq:loctilde}, this shows that
\begin{equation} \label{eq:loctildestrong}
	\lim_{N\to\infty} \; \max\Big\{ \bP_{N,\xi}(\tilde\cW_{1,N}) ,\; 
	\bP_{N,\xi}(\tilde\cW_{2,N}) \Big\}
	= 1 \,, \qquad \bbP(\dd\xi)\text{-almost surely}\,,	
\end{equation}
which yields \eqref{eq:1point} and, moreover, shows that
\begin{equation*}
	\rw_{N,\xi} \,=\, \begin{cases}
	z_N^{(1)} & \text{ if }
	\bP_{N,\xi} ( \tilde{\cW}_{1,N} ) > \bP_{N,\xi} (\tilde{\cW}_{2,N}) \\
	\rule{0pt}{1.4em}z_N^{(2)} & \text{ if }
	\bP_{N,\xi} ( \tilde{\cW}_{2,N} ) > \bP_{N,\xi} (\tilde{\cW}_{1,N})
	\end{cases} \,, \quad
	\text{eventually } \bbP(\dd\xi)\text{-almost surely}\,.
\end{equation*}

It is convenient to introduce some further notation. Recalling \eqref{eq:tildeWN},
for $N \in \N$ and $x \in \cB_N$ we define the following subsets of $\Omega_S$:
\begin{equation} \label{eq:tildeWNbis}
	\tilde \cW_N(x) \,:=\, \Big\{S \in \Omega_S:
	S_N = x, \ \ell_N(x) > \tfrac{N-|x|}{2}, \
	\ell_N(z) = 0 \ \forall z \text{ s.t. } 	\xi(z) > \xi(x) \Big\} \,,
\end{equation}
so that $\tilde\cW_{i,N} = \tilde\cW_N(z_N^{(i)})$. Next we set
\begin{equation} \label{eq:CN}
	C_N(x) \,:=\, \log \e \big[ e^{H_N(S)} \, \ind_{\{S \in \tilde\cW_N(x)\}} \big] \,,
\end{equation}
so that we can write
\begin{equation} \label{eq:useful}
	\big|\log \bP_{N,\xi} ( \tilde{\cW}_{1,N} ) 	- \log \bP_{N,\xi} (\tilde{\cW}_{2,N}) \big|
	\,=\, \big| C_N(z_N^{(1)}) - C_N(z_N^{(2)}) \big| \,.
\end{equation}
Finally, given an arbitrary $\epsilon \in (0, d/\alpha)$ and setting 
$N_k := \lfloor k^{2/\epsilon} \rfloor$,
we introduce the event $\cH_k \subseteq \Omega_\xi$ defined by
\begin{equation} \label{eq:Hk}
\begin{split}
	\cH_k \,:=\, \Big\{ & \xi \in \Omega_\xi: \ \exists x, y \in \cB_{N_{k+1}},
	\, x \ne y, \ \exists n \in \{\max\{|x|,|y|\}, \ldots, N_{k+1}\} \text{ such that} \\
	& \ \xi(x) > \tfrac{(N_k)^{d/\alpha}}{(\log N_{k+1})^{2/\alpha}}, \
	\xi(y) >	 \tfrac{(N_k)^{d/\alpha}}{(\log N_{k+1})^{2/\alpha}}, \
	| C_n(x) - C_n(y) | \le N_k^{d/\alpha - \epsilon} \Big\} \,.
\end{split}
\end{equation}

We are going to show that
\begin{equation} \label{eq:BCH}
	\sum_{k\in\N} \bbP(\cH_k) < \infty \,.
\end{equation}
We claim that this implies \eqref{eq:ratio} and completes the step. Indeed, by 
the Borel-Cantelli lemma it follows from \eqref{eq:BCH} that for $\bbP$-almost every
$\xi \in \Omega_\xi$ there exists $\overline k = \overline k(\xi) < \infty$
such that $\xi \not \in \cH_k$ for all $k \ge \overline k$. For any $N \ge N_{\overline k}$,
let $k \in \N, k \ge \overline k$ be such that $N_k < N \le N_{k+1}$ and note that, plainly,
$z_N^{(1)}, z_N^{(2)} \in \cB_N \subseteq \cB_{N_{k+1}}$.
Recalling the lower bound in \eqref{eq:lbZ1} and \eqref{eq:useful},
since $\xi \not \in \cH_k$ for all $k \ge \overline k$ we conclude that
eventually $\bbP(\dd\xi)$-almost surely
\begin{equation*}
	\big|\log \bP_{N,\xi} ( \tilde{\cW}_{1,N} ) 	- \log \bP_{N,\xi} (\tilde{\cW}_{2,N}) \big|
	\ge N^{d/\alpha - \epsilon} \,,
\end{equation*}
which is a stronger statement than \eqref{eq:ratio}.

\smallskip

We are left with proving \eqref{eq:BCH}, for which we have to estimate
\begin{equation} \label{eq:strum}
	\bbP \big( \xi(x) > t, \xi(y)>t, | C_n(x) - C_n(y)| \le M \big) \,,
\end{equation}
for suitable $t$ and $M$.
Recalling \eqref{eq:tildeWNbis} and \eqref{eq:CN}, it is useful to set
\begin{equation} \label{eq:CNyx}
	C_N(y;x) \,:=\, \log \e \big[ e^{H_N(S)} \, \ind_{S \in \tilde\cW_N(y)}
	\, \ind_{\{\ell_N(x)=0\}} \big] \,.
\end{equation}
Note in fact that, on the event $\xi(x) > \xi(y)$, we have $C_N(y) = C_N(y; x)$,
by the definition \eqref{eq:tildeWNbis} of $\tilde\cW_N(y)$. Therefore,
splitting \eqref{eq:strum} on $\{\xi(x) > \xi(y)\}$ and $\{\xi(x) < \xi(y)\}$
and using the symmetry between $x$ and $y$, we can easily estimate
\begin{equation} \label{eq:strum2}
\begin{split}
	& \bbP \big( \xi(x) > t, \xi(y)>t, | C_n(x) -  C_n(y)| \le M \big) \\
	& \quad \le 2 \, \bbP \big( \xi(x) > t, \xi(y)>t, 
	| C_n(x) -  C_n(y;x)| \le M \big) \\
	& \quad \le  2 \, \bbE \big[ \ind_{\{\xi(y) > t\}} \,
	\bbP \big( \xi(x) > t,\; | C_n(x) -  C_n(y;x)| \le M 
	\big| \cG_x \big) \big]\,,
\end{split}
\end{equation}
where $\cG_x := \sigma(\{\xi(z)\}_{z \in \Z^d \setminus\{x\}})$.
We stress that $ C_n(y;x)$ is $\cG_x$-measurable, because by
definition it does not depend on $\xi(x)$ (recall \eqref{eq:CNyx}).

We now need to study the dependence of $ C_n(x)$ on
$\xi(x)$ conditionally on $\cG_x$, i.e., when all the other
field variables $\{\xi(z), z \ne x\}$ are fixed. Recalling
\eqref{eq:CN}, \eqref{eq:tildeWNbis} and summing over the values
of the variable $\ell_N(x)$, we can write $C_n(x) = g(\xi(x))$,
where
\begin{equation*}
	g(s) \,:=\, \log \sum_{k = \frac 1 2 (n-|x| + 1)}^{n-|x|}
	e^{k s} \, c_{n,k} \qquad \text{and} \qquad
	c_{n,k} := \e \big[ e^{H_n(S) - k \xi(x)} \, \ind_{S \in \tilde \cW_n(x)}
	\, \ind_{\{\ell_N(x)=k\}} \big] \,.
\end{equation*}
We stress that, on the event $\{\ell_n(x)=k\}$, the term
$H_n(S) - k \xi(x)$ does not depend on $\xi(x)$. Therefore the coefficients $c_{n,k}$
(and, hence, the function $g(\cdot)$)
only depend on $\{\xi(z), z \ne x\}$, i.e., they are $\cG_x$-measurable.
Also note that the function $g(\cdot)$ is smooth and Lipschitz, since
\begin{equation*}
	g'(s) \,=\, \frac{\sum_{k = \frac 1 2 (n-|x| + 1)}^{n-|x|} k \, e^{ks} \, c_{n,k}}
	{\sum_{k = \frac 1 2 (n-|x| + 1)}^{n-|x|} e^{ks} \, c_{n,k}} \,\ge\,
	\frac 12 (n- |x| +1) \,.
\end{equation*}
Therefore, by the change of variables formula, from \eqref{eq:pareto} we obtain
\begin{equation*}
\begin{split}
	& \bbP\big( \xi(x) > t ,\ C_n(x) \in \dd v \,\big|\, \cG_x \big)
	= \bbP\big( \xi(x) > t ,\ g(\xi(x)) \in \dd v \,\big|\, \cG_x \big) \\
	& \quad = \ind_{\{g^{-1}(v) > \max\{1, t\}\}} \, \frac{1}{|g'(g^{-1}(v))|}
	\, \frac{\alpha}{(g^{-1}(v))^{1+\alpha}}
	\, \dd v \le
	\frac{2}{(n-|x|+1)} \, \frac{\alpha}{t^{1+\alpha}} \, \dd v \,,
\end{split}
\end{equation*}
hence
\begin{equation*}
\begin{split}
	& \bbP \big( \xi(x) > t,\; | C_n(x) -  C_n(y;x)| \le M \big| \cG_x \big) \\
	& \,=\, \bbP \big( \xi(x) > t,\; C_n(x) \in [C_n(y;x) - M,
	C_n(y;x) + M] \big| \cG_x \big) \,\le\, 
	\frac{2\alpha}{(n-|x|+1)t^{1+\alpha}} \cdot 2M \,.
\end{split}
\end{equation*}
Coming back to \eqref{eq:strum2}, since $\bbP(\xi(y) > t) \le t^{-\alpha}$, we conclude that
\begin{equation} \label{eq:strum3}
	\bbP \big( \xi(x) > t, \xi(y)>t, | C_n(x) -  C_n(y)| \le M \big) \,\le\,
	\frac{8 \alpha M}{(n-|x|+1)t^{1+2\alpha}} \,.
\end{equation}

We are finally ready to estimate $\bbP(\cH_k)$. Recalling the definition \eqref{eq:Hk}
and the fact that $N_k = \lfloor k^{2/\epsilon} \rfloor$, applying \eqref{eq:strum3} we obtain
\begin{equation*}
\begin{split}
	\bbP(\cH_k) & \,\le\, 2 \sumtwo{x \ne y \in \cB_{N_{k+1}}}{|x| \ge |y|} \;
	\sum_{n = |x|}^{N_{k+1}} \, \bbP \Big(
	\xi(x), \xi(y) > \tfrac{(N_k)^{d/\alpha}}{(\log N_{k+1})^{2/\alpha}}, 
	| C_n(x) -  C_n(y)| \le N_k^{d/\alpha - \epsilon} \big) \\
	& \,\le\, 2 \, (const.) \, (N_{k+1})^{2d} \, \frac{8\alpha N_k^{d/\alpha - \epsilon}}
	{\{ (N_k)^{d/\alpha} / (\log N_{k+1})^{2/\alpha}\}^{(1+2\alpha)}}
	\sum_{n = |x|}^{N_{k+1}} \frac{1}{n-|x|+1} \\
	& \,\le\, (const.') \, \frac{(\log N_{k+1})^{2/\alpha + 5}}{N_k^{\epsilon}}
	\,\le\, (const.'') \, \frac{(\log k^{2/\epsilon})^{2/\alpha + 5}}{k^2} \,,
\end{split}
\end{equation*}
from which \eqref{eq:BCH} follows. This completes the step.

\subsection{Proof of \eqref{eq:rweak}}
\label{sec:rweak}

In view of \eqref{eq:wz12}, it is sufficient to prove that
\begin{equation} \label{eq:rweakbis}
	\frac{z_{N}^{(1)}}{N} \, \Longrightarrow \, \rw \,, \qquad \text{where} \qquad
	\bbP( \rw \in \dd x) \,=\, c_\alpha \, (1-|x|)^\alpha \, \ind_{\{|x| \le 1\}} \,
	\, \dd x \,,
\end{equation}
and we recall that $c_\alpha := (\int_{|y| \le 1} (1-|y|)^\alpha \dd y)^{-1}$.

Setting $\phi_N(x) := 1 - \frac{|x|}{N+1}$ and recalling \eqref{eq:pareto},
for $x \in \cB_N$ and $t \in (1,\infty)$ we have
\begin{equation*}
\begin{split}
	\bbP \big( z_{N}^{(1)} = x \,, \ \xi(x) \in \dd t \big) & \,=\,
	\bbP \big( \xi(z) < t \ \forall z \in \cB_N \setminus\{ x\} \,, \ \xi(x) \in \dd t \big) \\
	& \,=\,
	\prod_{z \in \cB_N, \, z \ne x} \left( 1 - \frac{\phi_N(z)^\alpha}{t^\alpha \phi_N(x)^\alpha}
	\right) \, \frac{\alpha}{t^{1+\alpha}} \, \dd t \,,
\end{split}
\end{equation*}
therefore for all function $f: \R^d \to \R$ we can write
\begin{equation} \label{eq:outout}
\begin{split}
	\bbE \left[ f \left(\frac{z_N^{(1)}}{N}\right) \right]
	\,=\, \sum_{x\in\cB_N} f \left( \frac{x}{N} \right) \int_1^\infty \dd t
	\prod_{z \in \cB_N, \, z \ne x} \left( 1 - \frac{\phi_N(z)^\alpha}{t^\alpha \phi_N(x)^\alpha}
	\right) \, \frac{\alpha}{t^{1+\alpha}} \,.
\end{split}
\end{equation}
Now set $t = N^{d/\alpha} s$ and note that as $N \to\infty$,
uniformly in $s \in (\epsilon, \infty)$
and $x \in \cB_{(1-\epsilon)N}$, where $\epsilon > 0$ is arbitrary but fixed,
by a Riemann sum approximation we have
\begin{equation*}
\begin{split}
	\sum_{z \in \cB_N, \, z \ne x} \log \left( 1 - 
	\frac{\phi_N(z)^\alpha}{t^\alpha \phi_N(x)^\alpha} \right) & \,=\,
	- \frac{1}{s^\alpha (1 - \frac{|x|}{N+1})^\alpha N^d} \,
	\sum_{z \in \cB_N, \, z \ne x} 
	\left( 1 - \frac{|z|}{N+1} \right)^\alpha (1+o(1)) \\
	& \,=\, - \frac{c_\alpha^{-1}}{s^\alpha\, (1 - \frac{|x|}{N+1})^\alpha} \, (1+o(1)) \,.
\end{split}
\end{equation*}
Coming back at \eqref{eq:outout} and noting that
$\int_0^\infty \frac{\alpha}{s^{1+\alpha}} \, e^{-A/s^{\alpha}} \, \dd s = 
\int_0^\infty e^{-A u} \, \dd u = A^{-1}$, by a simple change of variables,
it follows again by a Riemann sum argument that if $f$ is continuous and bounded we have
\begin{equation*}
\begin{split}
	\lim_{N\to\infty} \bbE \left[ f \left(\frac{z_N^{(1)}}{N}\right) \right]
	& \,=\, \lim_{N\to\infty} \frac{1}{N^d} \sum_{x\in\cB_N} f \left( \frac{x}{N} \right)
	c_\alpha\, \left( 1-\frac{|x|}{N+1} \right)^{\alpha} \\
	& \,=\, c_\alpha \, \int_{|y| \le 1} f(y) \, (1-|y|)^\alpha \, \dd y \,,
\end{split}
\end{equation*}
proving \eqref{eq:rweakbis}.


\medskip
\section{Path properties}

\label{sec:paths}

In this section we prove Theorem~\ref{th:theopath}, i.e., we show
that $\lim_{N\to\infty}\bP_{N,\xi}(\cC_{N,\xi}) = 1$, $\bbP(\dd\xi)$-almost surely,
where the set $\cC_{N,\xi}$ is defined in \eqref{eq:defc}.

For $i = 1,2$, we denote for simplicity by $\tau_i :=
\inf\{n \in\N:\, S_n = z_N^{(i)}\}$ the first time at which the random walk
visits the site $z_N^{(i)}$ and we set
\begin{equation} \label{eq:DK}
\begin{split}
	\cD_{i,N} &\,:=\, \Big\{ S \in \Omega_S \colon\; \tau_i \le N\,, \
	S_m\neq S_n \ \forall m<n\leq  \tau_{i}\,, \
	S_n = z_N^{(i)} \ \forall n \in \{\tau_{i}, \ldots, N\} \Big\} \\
	\cK_{i,N} &\,:=\, \Big\{ S \in \Omega_S \colon\;
	\tau_{i} \le |z_{N}^{(i)}| + h_N \Big\} \,,
\end{split}
\end{equation}
where we recall that $h_N := (\log\log N)^{2/\alpha}\, N^{1-1/\alpha}$ if $\alpha > 1$
and $h_N := (\log N)^{1+2/\alpha}$ if $\alpha \le 1$.
Recalling the definition \eqref{eq:tildeWN} of the set $\tilde\cW_{i,N}$,
we are going show that for both $i=1,2$
\begin{gather} \label{eq:st1}
	\lim_{N\to\infty}\bP_{N,\xi} \big( \tilde\cW_{i,N} \setminus \cD_{i,N} \big) \,=\, 0\,,
	\qquad \bbP(\dd\xi)\text{-almost surely}\,, \\
	\label{eq:st2}
	\lim_{N\to\infty}\bP_{N,\xi} \big( ( \tilde\cW_{i,N} \cap \cD_{i,N})
	\setminus \cK_{i,N} \big) \,=\, 0\,,
	\qquad \bbP(\dd\xi)\text{-almost surely}\,.
\end{gather}
Recalling relation \eqref{eq:loctildestrong}, proved in the last section,
Theorem~\ref{th:theopath} is a consequence of \eqref{eq:st1} and \eqref{eq:st2}.
The rest of this section is therefore devoted to proving such relations.


\subsection{Step 1: proof of \eqref{eq:st1}}

We fix $i \in \{1,2\}$ throughout the section.
By definition, a random walk trajectory $S \in \cW_{i,N} \setminus \cD_{i,N}$ makes
either some \emph{loops} before time $\tau_i$ (i.e., before
reaching $z_N^{(i)}$) or some \emph{excursions} outside $z_N^{(i)}$ between time $\tau_i$
and time $N$.
We need to set up some notation to account for such loops and excursions.

We set $i_0 = j_0 := -1$ and, for $k \in \N$, we denote by 
$i_k = i_k(S)$, $j_k = j_k(S)$ the extremities of the 
$k$-th loop made by a trajectory $S \in \Omega_S$ before reaching $z_N^{(i)}$:
\begin{align}
\begin{split}
	i_k &:= \inf\big\{ n \in \{ j_{k-1} + 1, \ldots, \tau_i - 1\} \colon\; 
	\exists m \in \{n+1,\dots, \tau_i-1\} \text{ s.t. } S_m = S_n \big\} \,, \\
	j_k &:= \max\{ n < \tau_i \colon\; S_n = S_{i_k}\} \,,
\end{split}
\end{align}
with the usual convention $\inf \emptyset := \infty$.
We also set $\cI_k := \{i_{k}+1,\dots, j_{k}\}$ and $|\cI_k| := j_k - i_k$
for conciseness. Then we denote by
$\cN = \cN(S) := \max\{k \in \N: \, i_k < \infty\}$ the total number
and by $\cL = \cL(S) := \sum_{k=1}^\cN |\cI_k|$ the total length of the loops
of the trajectory~$S$. Note that $\cN = \cL = 0$ if $i_1 = \infty$, i.e.,
if the trajectory $S$ has no loops.
Finally, we denote by $\pi(S)$ the \emph{injective skeleton}
of $S$ before reaching $z_N^{(i)}$, i.e., the random walk trajectory of
$\tau_i - \cL$ steps defined (with some abuse of notation) by
\begin{equation}\label{def:Tn}
	\pi(S) = \{\pi(S)_n\}_{n \in\{0, \ldots, \tau_i - \cL\}}
	:= \{S_n\}_{n\in \{0,\dots,\tau_i\} \setminus \cup_{k=1}^{\cN} \cI_k}.
\end{equation}
We let $\cV_{i,N,r}$ denote the set of 
all $r$-steps injective paths, starting at $0$ and ending at $z_N^{(i)}$,
which do not visit any site $x\in \cB_N$ with $\xi(x)>\xi(z_N^{(i)})$
(recall \eqref{eq:lN}):
\begin{equation}
	\cV_{i,N,r} := \big\{ (S_n)_{n\leq r} \colon\;  
	S_r = z_N^{(i)},\  S_n \neq S_m \ \text{for} \ m \neq n, \
	\ell_r(x)=0 \ \text{when} \ \xi(x) > \xi(z_N^{(i)}) \big\} \,.
\end{equation}
Note that for $S \in \cW_{i,N} \setminus \cD_{i,N}$ we have
$\pi(S) \in \cV_{i,N,\tau_i - \cL(S)}$.

Next we deal with the excursions outside $z_N^{(i)}$.
Set $i'_0=j'_0=\tau_i-1$ and for $k\in\N$ denote by 
$i'_k = i'_k(S)$, $j'_k = j'_k(S)$ the extremities of the $k^{th}$
excursion outside $z_N^{(i)}$ made by the trajectory $S$ between time $\tau_i$
and time $N$:
\begin{align}
\begin{split}
	i'_k &:= \min\big\{ n \in \{j_{k-1}+1, \ldots, N-1\} \colon\; S_n \neq z_N^{(i)} \big\} \,,\\
	j'_k &:= \min\{ n > i'_k \colon\; S_n = z_N^{(i)}\} \,.
\end{split}
\end{align}
We also set $\cI'_k := \{i'_{k}+1, \dots, j'_{k}\}$ and
$|\cI'_k| := j'_k - i'_k$; furthermore, we
denote by $\cN' = \cN'(S) := \max\{k\geq 0 \colon\; i'_k < \infty\}$
the total number and by
$\cL' = \cL'(S) := \sum_{k=1}^{\cN'} |\cI'_k|$ the total length of the excursions
of the trajectory $S$. Note that $\cN' = \cL' = 0$ if
$i'_1 = \infty$, i.e., if there are no excursions.

We can now start with the proof of \eqref{eq:st1}.
Recalling the definition \eqref{eq:model} of our model
and using the notation we have just introduced, we obtain the decomposition
\begin{equation} \label{eq:bianconi}
	\bP_{N,\xi} \big( \tilde\cW_{i,N} \setminus \cD_{i,N} \big)
	\,=\, \frac{1}{U_{N,\xi}} \, \sum_{r=|z_N^{(i)}|}^N \sum_{S^* \in \cV_{i,N,r}}
	\e \big( e^{H_{N,\xi}(S)} \, \ind_{\{S \in \tilde\cW_{i,N} \setminus \cD_{i,N}\}}
	\, \ind_{\{\pi(S) = S^*\}} \big) \,.
\end{equation}
We bound the partition function $U_{N,\xi}$ from below by considering
the trajectories that reach $z_N^{(i)}$ through an injective path,
avoiding the sites $x$ with $\xi(x) > \xi(z_N^{(i)})$, and
stick at $z_N^{(i)}$ afterwards, getting
\begin{equation} \label{eq:lbU}
	U_{N,\xi} \,\ge\, \sum_{r=|z_N^{(i)}|}^N \sum_{S^* \in \cV_{i,N,r}}
	e^{\sum_{n=1}^{r-1} \xi(S^*_n) +
	(N+1-r) \xi(z_N^{(i)})} \, \p(S^*) \, \kappa^{N-r} \,,
\end{equation}
where for simplicity we set $\p(S^*) := \p(S_1 = S_1^*, \ldots, S_r = S_r^*)$
and we recall \eqref{eq:rw0}.

Next we estimate the double sum in the right hand side of \eqref{eq:bianconi}.
Observe that for $S \in \tilde\cW_{i,N} \setminus \cD_{i,N}$ we have
$\cL + \cL' \ge 1$, because $S$ must make at least one loop before reaching
$z_N^{(i)}$ or one excursion outside $z_N^{(i)}$ before time $N$.
By definition of $\cW_{i,N}$, cf. \eqref{eq:tildeWN}, any site $x$ visited
by $S$ in the loops or excursions has an associated potential $\xi(x) < \xi(z_N^{(i)})$,
hence $\xi(x) \le X_N^{(J_i + 1)} = \xi(z_N^{(i)}) -(X_N^{(J_i)} - X_N^{(J_i + 1)})$,
cf. \eqref{eq:defim}. It follows that on $\{\cL = l, \cL' = l'\}$ we have
$H_N(S) \le \sum_{n=1}^{r-1} \xi(S^*_n) + (N+1-r) \xi(z_N^{(i)}) - (l+l')(X_N^{(J_i)} - X_N^{(J_i + 1)})$, hence
\begin{equation*}
\begin{split}
	& \e \big( e^{H_{N,\xi}(S)} \, \ind_{\{S \in \tilde\cW_{i,N} \setminus \cD_{i,N}\}}
	\, \ind_{\{\pi(S) = S^*\}} \big) \\
	& \ \,\le\, \sum_{l, l' \in \N_0,\, l+l' \ge 1}
	e^{\sum_{n=1}^{r-1} \xi(S^*_n) + (N+1-r) \xi(z_N^{(i)}) - (l+l')(X_N^{(J_i)} - X_N^{(J_i + 1)})}
	\, \p(\cL = l, \cL' = l', \pi(S) = S^*) \,.
\end{split}
\end{equation*}
Looking back at \eqref{eq:bianconi} and \eqref{eq:lbU}, we conclude that
\begin{equation} \label{eq:dente}
\begin{split}
	& \bP_{N,\xi} \big( \tilde\cW_{i,N} \setminus \cD_{i,N} \big) \\
	& \ \,\le\, \suptwo{r \in \{|z_N^{(i)}|, \ldots, N\}}{S^* \in \cV_{i,N,r}} \
	\sum_{l, l' \in \N_0,\, l+l' \ge 1} e^{- (l+l')(X_N^{(J_i)} - X_N^{(J_i + 1)})} \,
	\frac{\p(\cL = l, \cL' = l', \pi(S) = S^*)}{\p(S^*) \, \kappa^{N-r}} \,.
\end{split}
\end{equation}

We are left with estimating the ratio in the right hand side of \eqref{eq:dente}.
It is convenient to disintegrate the
event $\{\cL = l\}$ (resp. $\{\cL' = l'\}$) by summing on the total number
$\cN$ and the locations $\cI = \{\cI_k\}_{k \le \cN}$ of the loops
(resp. the total number
$\cN'$ and the locations $\cI' = \{\cI'_k\}_{k \le \cN}$ of the excursions).
Using the Markov property and bounding the probability of each loop
and excursion (trivially) by $1$, for all $n, I = \{I_k\}_{k \le n}$,
$n', I' = \{I'_k\}_{k \le n}$ and
for all injective trajectories $S^* \in \cV_{i,N,r}$ we have
\begin{equation*}
	\p(\cN = n, \cI = I, \cN' = n', \cI' = I', \pi(S) = S^*)
	\le \p(S^*) \, \kappa^{N-r-l-l'} \,,
\end{equation*}
because $|\{n \in \{\tau_i, \ldots, N-1\}: S_n = S_{n+1}\}|
= N - \tau_i - \cL'$, by definition of $\cL'$,
and $\tau_i = r + \cL$ when $\pi(S) = S^* \in \cV_{i,N,r}$,
by definition of $\cL$. It follows that
\begin{equation*}
	\frac{\p(\cL = l, \cL' = l', \pi(S) = S^*)}{\p(S^*) \, \kappa^{N-r}}
	\le \kappa^{-l-l'} \cdot \textstyle \big| \big\{ (n, I, n', I'): 
	\, \sum_{k=1}^n |I_k| = l,\, \sum_{k=1}^{n'} |I'_k| = l' \big\} \big| \,.
\end{equation*}
It remains to bound the cardinality of the set in the right hand side.
For fixed $n \in \{0, \ldots, l\}$, the intervals $I = \{I_k\}_{k \le n}$ consist of $2n$
points in $\{0, \ldots, \tau_i\} \subseteq \{0, \ldots, N\}$, therefore
the number of possible choices for $I$ is bounded from above by $(N+1)^{2n} \le (N+1)^{2l}$.
Analogously, for every $n' \in \{0, \ldots, l'\}$,
the number of choices for $I'$ is bounded from above by
$(N+1)^{2n'} \le (N+1)^{2l'}$. Looking back at \eqref{eq:dente}, we can write
\begin{equation*}
\begin{split}
	\bP_{N,\xi} \big( \tilde\cW_{i,N} \setminus \cD_{i,N} \big) & \,\le\,
	\sum_{l, l' \in \N_0,\, l+l' \ge 1} 
	e^{- (l+l') ( X_N^{(J_i)} - X_N^{(J_i + 1)} + \log \kappa
	- 2 \log (N+1) )} \, (l+1) \, (l'+1) \\
	& \,\le\, (const.) \sum_{m=1}^\infty 
	e^{- m ( X_N^{(J_i)} - X_N^{(J_i + 1)} + \log \kappa
	- 2 \log (N+1) )} \, m^3\\
	& \,\le\, (const.') \, \frac{e^{- ( X_N^{(J_i)} - X_N^{(J_i + 1)} + \log \kappa
	- 2 \log (N+1) )}}{\big( 1 - e^{- ( X_N^{(J_i)} - X_N^{(J_i + 1)} + \log \kappa
	- 2 \log (N+1) )}\big)^4} \,,
\end{split}
\end{equation*}
where in the second inequality we have used that
$\sum_{l, l' \in \N_0:\, l + l' = m} (l+1)(l'+1) \le (const.) m^3$.
It then follows from Corollary~\ref{th:boundJ} and Proposition~\ref{th:gap15}
that relation \eqref{eq:st1} holds true, completing the first step.

\subsection{Step 2: proof of \eqref{eq:st2}}

Throughout the section we fix $i \in \{1,2\}$. We recall that $\tau_i :=
\inf\{n \in\N:\, S_n = z_N^{(i)}\}$ denotes the first time at which the random walk
visits $z_N^{(i)}$. 

A random walk trajectory $S \in \tilde\cW_{i,N} \cap \cD_{i,N}$
(cf. \eqref{eq:tildeWN} and \eqref{eq:DK}) reaches $z_N^{(i)}$ through an injective
path, avoiding sites where the potential is larger than $\xi(z_N^{(i)})$, and
sticks at $z_N^{(i)}$ afterwards (from time $\tau_i$ to time $N$). 
Therefore the corresponding Hamiltonian (cf. \eqref{eq:model}) is bounded from above by
\begin{equation*}
	H_{N,\xi}(S) \,\le\, \sum_{n=1}^{\tau_i-1} \xi(S_i) \,+\,
	(N+1-\tau_i) \xi(z_N^{(i)}) \,\le\, \sum_{j=1}^{N} X_N^{(j)} \,+\,
	(N+1-\tau_i) \xi(z_N^{(i)}) \,.
\end{equation*}
Recalling the definition \eqref{eq:DK} of the set $\cK_{i,N}$, for
$S \in (\tilde\cW_{i,N} \cap \cD_{i,N}) \setminus \cK_{i,N}$
we obtain
\begin{equation*}
	H_{N,\xi}(S) \,\le\, \sum_{j=1}^{N} X_N^{(j)} \,+\,
	\big( N+1-|z_N^{(i)}| - h_N \big) \xi(z_N^{(i)}) \,,
\end{equation*}
therefore, cf. \eqref{eq:model},
\begin{equation*}
	\bP_{N,\xi} \big( (\tilde\cW_{i,N} \cap \cD_{i,N}) \setminus \cK_{i,N} \big)
	\,\le\, \frac{1}{U_{N,\xi}} \, e^{\sum_{j=1}^{N} X_N^{(j)} \,+\,
	( N+1-|z_N^{(i)}| - h_N ) \xi(z_N^{(i)})} \,.
\end{equation*}
As usual, we obtain a lower bound on $U_{N,\xi}$ by considering a single trajectory
that reaches the site $z_N^{(i)}$ in $|z_N^{(i)}|$ steps and sticks there afterwards, getting
\begin{equation*}
	U_{N,\xi} \,\ge\, e^{(N+1-|z_N^{(i)}|) \xi(z_N^{(i)})} \, c^N \,,
\end{equation*}
for a suitable $c > 0$, cf. \eqref{eq:rw0}. Note that $\xi(z_N^{(i)}) \ge
Z_N^{(i)} \ge N^{d/\alpha}/(\log \log N)^{3/2\alpha}$
eventually $\bbP(\dd\xi)$-almost surely, for both $i \in \{1,2\}$, by
relation \eqref{eq:lbZ1}. Therefore
\begin{equation*}
	\bP_{N,\xi} \big( (\tilde\cW_{i,N} \cap \cD_{i,N}) \setminus \cK_{i,N} \big)
	\,\le\, e^{\sum_{j=1}^{N} X_N^{(j)} \,-\, h_N \, N^{d/\alpha} /(\log \log N)^{3/2\alpha}} \,.
\end{equation*}
Since $h_N := (\log\log N)^{2/\alpha}\, N^{1-1/\alpha}$ if $\alpha > 1$
and $h_N := (\log N)^{1+2/\alpha}$ if $\alpha \le 1$, it follows from \eqref{eq:sumbou}
that $\bP_{N,\xi} \big( (\tilde\cW_{i,N} \cap \cD_{i,N}) \setminus \cK_{i,N} \big) \to 0$
as $N\to\infty$, $\bbP(\dd\xi)$-almost surely. This proves that \eqref{eq:st2}
holds true and completes the second step.


\newcommand{\znk}{Z_N^{(k)}}
\newcommand{\znu}{Z_N^{(1)}}
\newcommand{\umd}{1 -\delta}
\newcommand{\umdp}{(1 -\delta)}
\def\prob#1{{{\bbP}\left ( {#1} \right )}}
\def\etp#1{{\left ( {#1} \right )}}
\def\esp#1{{\bbE\left [ {#1} \right ]}}
\def\valabs#1{{\left | {#1} \right |}}
\def\ens#1{{\left \{ {#1} \right \}}}
\def\un#1{1_{\ens{#1}}}
\def\unsur#1{\frac{1}{{#1}}}
\let\Z=\bbZ
\newcommand{\mac}{M_{A^c}}

\bigskip

\appendix

\section{Order statistics for the field}
\label{ap:field}


This section is devoted to the order statistics
$X_N^{(1)}, \ldots, X_N^{(|\cB_N|)}$ of the field
$\{\xi(x)\}_{x \in \cB_N}$.
We first give some basic probability estimates,
from which the proofs of Lemma~\ref{th:ubX1}
and Proposition~\ref{th:gap15} will be deduced.


\subsection{Basic estimates}

We start comparing the relative sizes of $X_N^{(k)}$ and $X_N^{(p)}$.


\begin{lemma}\label{le:lem}
For all $N,p,k \in \N$ with $1 \le p < k \le |\cB_N|$ and for all
$\delta \in (0,1)$ we have
\begin{equation} \label{eq:Xk}
        \bbP \big( X_N^{(k)}\geq (1-\delta) X_N^{(p)} \big) \le
        \binom{k-1}{k-p} \, (1-(1-\delta)^\alpha)^{k-p} \,.
\end{equation}
In the special case $p=1$ the equality holds:
\begin{equation} \label{eq:Xk1}
        \bbP \big( X_N^{(k)}\geq (1-\delta) X_N^{(1)} \big) =
        (1-(1-\delta)^\alpha)^{k-1} \,.
\end{equation}
\end{lemma}

\begin{proof}
We introduce the shortcuts $M_A := \sup_{x\in A} \xi(x)$,
$\{X_N^{(m) \cdots (n)}\} := \{X_N^{(m)}, \ldots, X_N^{(n)}\}$
and $A^c := \cB_N \setminus A$ for convenience.
We recall that $\cB_N = \{z \in \Z^d : \valabs{z} \le N\}$.
Summing over the location of the subsets
$\{X_N^{(1) \cdots (k-1)}\} = A$ and $\{X_N^{(p) \cdots (k-1)}\} = B$,
so that $X_N^{(k)} = M_{A^c}$ and
$X_N^{(p)} = M_{(A \setminus B)^c}$, we can write
\begin{align*}
        & \bbP(X_N^{(k)}\geq (1-\delta) X_N^{(p)} )  \\
        & \;=\;         
	\sumtwo{A \subseteq \cB_N, \, |A| = k-1}{B \subseteq A,\, |B|=k-p}
        \bbP \left( X_N^{(k)}\geq (1-\delta) X_N^{(p)} , \
	\{X_N^{(1) \cdots (k-1)}\} = A, \ 
	\{X_N^{(p) \cdots (k-1)}\} = B \right) \\
        & \;=\;         
	\sumtwo{A \subseteq \cB_N, \, |A| = k-1}{B \subseteq A,\, |B|=k-p}
        \bbP \left( M_{A^c} < \xi(y) < \frac{1}{1-\delta} M_{A^c} \ \forall y \in B \,,
        \ \ \xi(z) > M_B \ \forall z \in A \setminus B \right) \,.
\end{align*}
Since $M_B \ge M_{A^c}$ on the event we are considering,
we can replace $M_B$ by $M_{A^c}$ and obtain the upper bound
\begin{align*}
        \bbP(X_N^{(k)}\geq (1-\delta) X_N^{(p)} )
        \;\le\; \sumtwo{A \subseteq \cB_N, \, |A| = k-1}
	{B \subseteq A,\, |B|=k-p} \bbP \bigg( \frac{(1-\delta)^\ga}
	{(M_{A^c})^\ga} < \frac{1}{\xi(y)^\ga}
        <  \frac{1}{(M_{A^c})^\ga} \ \forall y \in B \,, \ \ & \\
        \frac{1}{\xi(z)^\ga} < \frac{1}{(M_{A^c})^\ga} 
        \ \forall z \in A \setminus B & \bigg) \,.
\end{align*}
We stress that in the special case $p=1$ we have $A = B$,
so that $A \setminus B = \emptyset$ and therefore the
above inequality is an equality.

By assumption the field $\xi(\cdot)$ has a Pareto distribution 
with parameter $\alpha >0$, cf. \eqref{eq:pareto},
therefore $\frac{1}{\xi^\alpha}$ 
is uniformly distributed on the interval $(0,1)$:
$\bbP(a < \frac{1}{\xi} < b) = b-a$ for all $0 < a < b  <1$.
It follows that
\begin{align*}
        \bbP(X_N^{(k)} \geq & (1-\delta) X_N^{(p)}) \;\le\;
        (1-(1-\delta)^\ga)^{k-p} \, 
        \sumtwo{A \subseteq \cB_N, \, |A| = k-1}{B \subseteq A,\, |B|=k-p}
        \bbE \left( \frac{1}{(M_{A^c})^{\ga(k-1)}} \right) \\
        & \;\le\; \binom{k-1}{k-p} \, (1-(1-\delta)^\ga)^{k-p} \, 
        \sum_{A \subseteq \cB_N, \, |A| = k-1}
        \bbE \left( \frac{1}{(M_{A^c})^{\ga(k-1)}} \right) \,,
\end{align*}
and again all these inequalities are equalities if $p=1$.
It only remains to check that the last sum equals one.
To this purpose, note that for all $\ell \in \N$,
summing on the location of the set $\{X_N^{(1) \cdots (\ell)}\}$,
we can write
\begin{align*}
        1 & \;=\; \sum_{A \subseteq \cB_N\,, \, |A|=\ell}
        \bbP( \{X_N^{(1) \cdots (\ell)}\} = A ) \;=\;
        \sum_{A \subseteq \cB_N\,, \, |A|=\ell}
        \bbP( \xi(x) > M_{A^c} \ \forall x \in A ) \\
        & \;=\; \sum_{A \subseteq \cB_N\,, \, |A|=\ell}
        \bbP \left( \frac{1}{\xi(x)^\ga} < \frac{1}{(M_{A^c})^\ga} \ \forall x \in A \right)
        \;=\; \sum_{A \subseteq \cB_N\,, \, |A|=\ell}
        \bbE \left( \frac{1}{(M_{A^c})^{\ga\ell}} \right).\qedhere
\end{align*}
\end{proof}

\medskip

Next we give some bounds on the absolute size of $X_N^{(k)}$.

\begin{lemma}
Let $c, C > 0$ be such that $c \le \frac{|\cB_N|}{N^d} \le C$.
Then for all $k \in \{1, \ldots, |\cB_N|\}$ and $t \in (0,\infty)$ 
the following relations hold:
\begin{gather} \label{eq:XN1t}
        \bbP(X_N^{(k)} > N^{d/\ga} t) \;\le\; \frac{C^k}{(k-1)!}
	\, \frac{1}{t^{k\ga}} \,, \\
	\label{eq:XNkt}
        \bbP(X_N^{(k)} \le t N^{d/\ga}) \;\le\; e^{-\frac{c}{t^\ga}} 
        \, \sum_{m=0}^{k-1} \frac{1}{m!} \, 
        \left(\frac{eC}{t^{\ga}}\right)^m \,.
\end{gather}
\end{lemma}

\begin{proof}
Throughout the proof we shall assume
that $t \ge N^{-d/\alpha}$.
In fact, for $t < N^{-d/\alpha}$ there is nothing to prove,
because the left hand side of \eqref{eq:XNkt} is zero
(recall that the field $\xi(\cdot)$
is bounded from below by one, cf. \eqref{eq:pareto})
and the right hand side of \eqref{eq:XN1t}
is greater than one: in fact, for $k \le |\cB_N|$ we have
$(k-1)! \le k^k \le |\cB_N|^k \le (C N^{d})^k$ 
and therefore for $t < N^{-d/\alpha}$
\begin{equation*}
	\frac{C^k}{k!} \, \frac{1}{t^{k\ga}} \ge
	\frac{C^k}{(C N^{d})^k} \, \frac{1}{t^{k\ga}}
	= \frac{1}{(N^{d/\alpha} t)^\alpha} \ge 1 \,.
\end{equation*}

We start proving \eqref{eq:XN1t}. The case $k=1$ is easy:
\begin{equation*}
        \bbP(X_N^{(1)} \le N^{d/\ga} t) \;=\; \bbP( \xi(x) \le N^{d/\ga} t \ \forall x \in \cB_N)
        \;=\; \left(1 - \frac{1}{t^\ga N^d} \right)^{|\cB_N|} \,,
\end{equation*}
and since $(1-z)^a \ge 1 - az$ for $a \ge 1$ and $z \in [0,1]$ we obtain
\begin{equation} \label{eq:uno}
        \bbP(X_N^{(1)} > N^{d/\ga} t) \;=\; 
	1 - \left(1 - \frac{1}{t^\ga N^d} \right)^{|\cB_N|}
        \; \le \; \frac{|\cB_N|}{N^d} \frac{1}{t^\ga}
	\;\le\; \frac{C}{t^\alpha} \,.
\end{equation}
For the general case, summing over the location of the set
$\{X_N^{(1) \cdots (k-1)}\} := \{X_N^{(1)}, \ldots, X_N^{(k-1)}\}$
and recalling the shortcuts $M_A := \sup_{x\in A} \phi(x)$
and $A^c := \cB_N \setminus A$ we get
\begin{equation*}
\begin{split}
        \bbP(X_N^{(k)} > N^{d/\ga} t) & \;=\; 
	\sum_{A \subseteq \cB_N,\, |A| = k-1}
	\bbP(X_N^{(k)} > N^{d/\ga} t , \
	\{X_N^{(1) \cdots (k-1)}\} = A) \\
	& \;=\; 
	\sum_{A \subseteq \cB_N,\, |A| = k-1}
	\bbP(M_{A^c}  > N^{d/\alpha} t, \
	\xi(x) > M_{A^c} \ \forall x \in A) \\
	& \;=\; 
	\sum_{A \subseteq \cB_N,\, |A| = k-1}
	\bbP \left( M_{A^c}  > N^{d/\alpha} t, \
	\frac{1}{\xi(x)^\alpha} < \frac{1}{M_{A^c}^\alpha}
	\ \forall x \in A \right) \,.
\end{split}
\end{equation*}
We have already remarked that the random variables
$1/\xi(x)^\alpha$ are uniformly distributed over the interval
$(0,1)$, that is $\bbP(\frac{1}{\xi(x)^\alpha} \le s ) = s$ for
$s \in (0,1)$. Then with some easy bounds we obtain
\begin{equation*}
\begin{split}
        & \bbP ( X_N^{(k)} > N^{d/\ga} t) \;=\; 
	\sumtwo{A \subseteq \cB_N}{|A| = k-1}
	\bbE \left( \frac{1}{M_{A^c}^{\alpha(k-1)}} \,,
	M_{A^c}  > N^{d/\alpha} t \right) \\
	& \ \;\le\; 
	\frac{1}{N^{d(k-1)} t^{\alpha(k-1)}}
	\sumtwo{A \subseteq \cB_N}{|A| = k-1}
	\bbP ( M_{A^c}  > N^{d/\alpha} t )
	\;\le\; 
	\frac{1}{N^{d(k-1)} t^{\alpha(k-1)}}
	\sumtwo{A \subseteq \cB_N}{|A| = k-1}
	\bbP ( X_N^{(1)}  > N^{d/\alpha} t ) \,.
\end{split}
\end{equation*}
where we have used that $\bbP ( M_{A^c}  > N^{d/\alpha} t ) \le 
\bbP ( X_N^{(1)}  > N^{d/\alpha} t )$ for all $A \subseteq \cB_N$.
Since $\binom{n}{m} \le n^m/m!$ and $|\cB_N|\le C N^d$, we obtain
\begin{equation*}
\begin{split}
        \bbP ( X_N^{(k)} > N^{d/\ga} t) & \;\le\; 
	\frac{1}{N^{d(k-1)} t^{\alpha(k-1)}} \,
	\binom{|\cB_N|}{k-1} \,
	\bbP ( X_N^{(1)}  > N^{d/\alpha} t ) \\
	& \;\le\; \frac{1}{N^{d(k-1)} t^{\alpha(k-1)}} 
	\, \frac{|\cB_N|^{k-1}}{(k-1)!}
	\, \frac{C}{t^\alpha} \le
	\frac{C^k}{(k-1)!} \, \frac{1}{t^{\alpha k}} \,,
\end{split}
\end{equation*}
having applied \eqref{eq:uno}. Equation \eqref{eq:XN1t} is proved.

To prove \eqref{eq:XNkt}, note that the random variable
$Y := \#\{z \in \cB_N:\, \xi(z) > t N^{d/\ga} \}$ is binomial
$B(n,p)$ with parameters $n=|\cB_N|$ and 
$p = \bbP(\xi > t N^{d/\ga}) =  1/(t^\ga N^d)$, therefore
\begin{equation} \label{eq:baraba}
\begin{split}
        \bbP(X_N^{(k)} \le t N^{d/\ga}) &\;=\; \bbP(Y \le k-1) \;=\;
        \sum_{m=0}^{k-1} \binom{n}{m} p^m \, (1-p)^{n-m} \\
        & \;=\; \sum_{m=0}^{k-1} \binom{|\cB_N|}{m} 
	\left( \frac{1}{t^\ga N^d} \right)^m
        \, \left( 1 - \frac{1}{t^\ga N^d} \right)^{|\cB_N|-m} \,.
\end{split}
\end{equation}
Using the estimates $(1-x)^a \le e^{-ax}$ and $\binom{n}{m} \le n^m/m!$ we get
\begin{align*}
	\bbP(X_N^{(k)} \le t N^{d/\ga}) &\;\le\; 
	e^{-\frac{|\cB_N|}{N^d} \frac{1}{t^\ga}}
        \sum_{m=0}^{k-1} \frac{1}{m!} \, \frac{1}{t^{\ga m}}
        \left( \frac{|\cB_N|}{N^d} \, e^{\frac1{t^\ga N^d}} \right)^{m} \,,
\end{align*}
from which \eqref{eq:XNkt} follows, recalling that
$|\cB_N| \ge c N^d$ and $1/(t^\ga N^d) \le 1$ by assumption.
\end{proof}

\medskip

We are finally ready for the proof of Lemma~\ref{th:ubX1}
and Proposition~\ref{th:gap15}, to which are devoted
the next paragraphs. For convenience, the proof of Lemma~\ref{th:ubX1}
has been split in two parts,
in which we consider each equation separately.


\smallskip

\subsection{Proof of Lemma~\ref{th:ubX1}}
\label{sec:ubX1proof}

We start considering equation \eqref{eq:ubX1}.
Let us set $N_k := 2^k$. By \eqref{eq:XN1t} we have
\begin{equation*}
        \sum_{k\in\N} \bbP(X_{N_k}^{(1)} > (N_k)^{d/\ga}
        \, (\log N_k)^{1/\ga+\gep/2}) \le \frac{C}{(\log 2)^{1+\ga \gep/2}}
        \sum_{k\in\N} \frac{1}{k^{1+\ga \gep/2}} < \infty \,,
\end{equation*}
and by \eqref{eq:XNkt}
\begin{equation*}
\begin{split}
	\sum_{k\in\N}
        \bbP(X_{N_k}^{(1)} \le (N_k)^{d/\ga}
        \, (\log \log N_k)^{-1/\alpha -\gep/2})
        & \le \sum_{k\in\N}
	\exp\left( - c (\log \log N_k)^{1 + \gep\alpha/2} \right) \\
	& = \sum_{k\in\N} \frac{1}{(k \, \log 2)^{c (\log \log 2 + \log k)^{\epsilon \alpha/2}}} 
		< \infty \,,
\end{split}
\end{equation*}
because for large $k$ the exponent $c (\log \log 2 + \log k)^{\epsilon \alpha/2}$
exceeds $1$.
By the Borel-Cantelli lemma, it follows that eventually (in $k$) $\bbP$-a.s.
\begin{equation} \label{eq:outil}
        \frac{(N_k)^{d/\ga}}{(\log\log N_k)^{1/\alpha + \epsilon/2}} \le
	X_{N_k}^{(1)} \le (N_k)^{d/\ga} \,
        (\log N_k)^{1/\ga + \gep/2} \,.
\end{equation}
Now take a generic $N\in\N$ and set
$k := \lfloor \log_2(N) \rfloor$, so that 
$N_k \le N < N_{k+1}$. Observe that
$X_{N_k}^{(1)} \le X_N^{(1)} \le X_{N_{k+1}}^{(1)}$,
because $X_N^{(1)}$ is increasing in $N$.
Plainly, one has $N_{k+1} \le 2 N$,
$N_k \ge \frac 12 N$, $\log N_k \le \log N$ and
$\log N_{k+1} \le \log 2 + \log N \le 2 \log N$ (for large $N$).
Then it follows from \eqref{eq:outil} that for large $N$
\begin{equation*}
        2^{- d/\alpha} \frac{N^{d/\alpha}}{(\log\log N)^{\epsilon/2}}
	\le X_{N_k}^{(1)} \le
        X_N^{(1)} \le X_{N_{k+1}}^{(1)} \le 2^{d/\ga + 1/\ga + \gep/2} \,
        N^{d/\ga} \, (\log N)^{1/\ga + \gep/2}\,.
\end{equation*}
Equation \eqref{eq:ubX1} follows observing that
$2^{d/\alpha} \le (\log \log N)^{\epsilon/2}$ and
$2^{d/\ga + 1/\ga + \gep/2} \le (\log N)^{\epsilon/2}$ for large $N$.

Next we focus on the lower bound in equation \eqref{eq:asXklogN}.
By \eqref{eq:XNkt} we can write
\begin{equation*}
        \bbP\left(X_N^{((\log N)^\theta)} 
	\le \frac{N^{d/\ga}}{(\log N)^{\theta/\alpha + \epsilon}} 
	\right) \;\le\;
        e^{-c \, (\log N)^{\theta + \alpha \epsilon}} 
        \, \sum_{m=0}^{\lfloor (\log N)^\theta\rfloor - 1} 
        \frac{1}{m!} \, \left(e C \, (\log N)^{\theta + \alpha \epsilon} 
	\right)^m \,.
\end{equation*}
Observe that, for fixed $x > 0$, the sequence $m \mapsto x^m/m!$ is increasing for $m \le x$,
therefore for $k \le x$ we have $\sum_{m=0}^{k-1} x^m/m! \le k x^k/k! \le k (e x/k)^k$,
because $m! \ge (m/e)^m$ for all $m\in\N$. It follows that
for some constant $C' > 0$ and for large $N$ we can write
\begin{equation} \label{eq:BCLB}
\begin{split}
        \bbP\bigg(X_N^{((\log N)^\theta)} \le &
	\frac{N^{d/\ga}}{(\log N)^{\theta/\alpha + \epsilon}} \bigg)
        \;\le\; e^{-c \, (\log N)^{\theta + \alpha \epsilon}} 
	\, (\log N)^\theta \,
        \big( C' \, (\log N)^{\ga \epsilon} \big)^{(\log N)^\theta} \\
        & \qquad \qquad \;\le\; (\log N)^\theta \, 
	e^{-c \, (\log N)^{\theta + \alpha \epsilon} \,
	+\, (\log N)^\theta [ \alpha \epsilon \log \log N 
	+ \log C']} \\
        & \qquad \qquad \;\le\; (\log N)^\theta \, e^{-\frac 12 c \, 
	(\log N)^{\theta + \alpha \epsilon}} \;\le\; N^{-2} \,,
\end{split}
\end{equation}
because by assumption $\theta > 1$ and $\epsilon > 0$
(the $-2$ could be replaced by any negative number).
The Borel-Cantelli lemma then yields 
directly the lower bound in \eqref{eq:asXklogN}.

Finally, we prove together the upper bound in \eqref{eq:asXklogN} and \eqref{eq:asbig}. 
By Stirling's formula we have 
$(k-1)! \ge (\frac{k-1}{e})^{k-1} \ge (\frac{k}{3})^{k}$ for large $k$.
Applying \eqref{eq:XN1t}, we can then write
\begin{equation*}
	\bbP\left(X_N^{(k)} > A \, \frac{N^{d/\ga}}{k^{1/\alpha}} 
	\right) \;\le\; \frac{C^{k}}	{(k - 1)!}
	\, \left( \frac{k^{1/\alpha}}{A} \right)^{\alpha k}
	\;\le\; \left( \frac{3C}{A^\alpha} \right)^k \;\le\; e^{-2 k} \,,
\end{equation*}
provided $A$ is chosen larger than $(e^2 / 3C)^{1/\alpha}$.
By the inclusion bound,
\begin{equation*}
	\bbP\left( \exists k \in \{(\log N), \ldots, |\cB_N|\}:
	\ X_N^{(k)} > A \, \frac{N^{d/\ga}}{k^{1/\alpha}} \right)
	\;\le\; \sum_{k \ge \log N} e^{-2k} \;\le\; \frac{(const.)}{N^2} \,,
\end{equation*}
therefore by the Borel-Cantelli lemma it follows that, eventually
$\bbP$-almost surely in $N$, one has
$X_N^{(k)} \le A \, \frac{N^{d/\ga}}{k^{1/\alpha}}$ for all $k \ge \log N$.
This yields immediately \eqref{eq:asbig}, as well as the upper bound 
in \eqref{eq:asXklogN}, because by assumption $\theta > 1$.
%
\qed


%
%


\smallskip

\subsection{Proof of Proposition~\ref{th:gap15}}
\label{sec:gap15proof}

Since the relation \eqref{eq:gap15} becomes stronger
as $\beta$ increases, we can safely assume that $\beta > 1$.
Then by \eqref{eq:asXklogN} we have that, eventually $\bbP$-a.s.,
\begin{equation} \label{eq:lowinc}
	X_N^{(k)} \ge X_N^{((\log N)^{\beta})}
	\ge \frac{N^{d/\alpha}}{(\log N)^{2\beta/\alpha}}\,,
	\qquad \forall k \le (\log N)^\beta \,.
\end{equation}
Since a more quantitative control will be needed later,
we observe that for large $N$
\begin{equation} \label{eq:Cn}
	\bbP(\cC_N) \;\le\; \frac{1}{N} \,,
	\qquad \text{where} \qquad
	\cC_N \;:=\; \union_{m \ge N}
	\left\{ X_m^{((\log m)^{\beta})}
	\le \frac{m^{d/\alpha}}{(\log m)^{2\beta/\alpha}} \right\} \,,
\end{equation}
as it follows from \eqref{eq:BCLB}.

Thanks to \eqref{eq:lowinc},
in order to prove \eqref{eq:gap15} it suffices to show
that for every $\beta > 1$ there exists $\gamma > 0$
such that, eventually $\bbP$-a.s.,
the following event holds:
\begin{equation*}
	\cV_N := \left\{ \forall k\leq (\log N)^\beta \colon\, 
	X_{N}^{(k)}-X_{N}^{(k+1)} \ge 
	\frac{X_{N}^{(k)}}{(\log N)^\gamma} \right\} \,.
\end{equation*}
In order to apply the Borel-Cantelly lemma, it is convenient
to group the events $\cV_N$ together. More precisely,
for $n\in\N_0$ we set
$N_n := \lfloor e^{n^r} \rfloor$, where the constant 
$r \in (0,1)$ will be fixed later, and we define
\begin{equation*}
	\tilde\cV_n := \bigcap_{N_{n} < m \le N_{n+1}} \cV_m \,.
\end{equation*}
The proof is then completed once we show that
the event $\tilde\cV_n$ holds eventually $\bbP$-a.s. (in~$n$).

It only remains to show that $\bbP(\tilde\cV_n^c)$ decays fast enough
as $n\to\infty$.
By construction, if $\tilde{\cV}_n$ does not hold,
there must exist $m \in \{N_{n}+1, \ldots, N_{n+1}\}$
and $k \le (\log m)^\beta$ such that 
$0 < X_m^{(k)} - X_m^{(k+1)} < (\log m)^{-\gamma} X_m^{(k)}$.
Let $y, z \in \cB_m$ be the two points at which the values
$X_m^{(k)}$ and $X_m^{(k+1)}$ are attained, that is
$\xi(y) = X_m^{(k)}$ and $\xi(z) = X_m^{(k+1)}$.
It is convenient to distinguish three cases, according to whether
$y$ and $z$ are in $\cB_{N_{n}}$ or not.
\begin{enumerate}
\item If both $y, z \in \cB_{N_{n}}$,
we can write $\xi(y) = X_{N_n}^{(k')}$ and $\xi(z) = X_{N_n}^{(k'')}$
for some $k' < k''$. Since by construction
$\xi(z) = X_m^{(k+1)}$ and $\cB_m \supseteq \cB_{N_n}$,
we must have $k'' \le k+1$, whence
$k' \le k \le (\log m)^\beta \le (\log N_{n+1})^\beta$.
Also note that
\begin{equation*}
\begin{split}
X_{N_n}^{(k')} - X_{N_n}^{(k'+1)} & \le
X_{N_n}^{(k')} - X_{N_n}^{(k'')} =
X_m^{(k)} - X_m^{(k+1)} \\
& < (\log m)^{-\gamma} X_m^{(k)}
= (\log m)^{-\gamma} X_{N_n}^{(k')}
\le (\log N_n)^{-\gamma} X_{N_n}^{(k')} \,.
\end{split}
\end{equation*}
This shows that, if $\tilde{\cV}_n$ does not hold
and both $y, z \in \cB_{N_{n}}$, there must exist
$k' \le (\log N_{n+1})^\beta$ such that
$X_{N_n}^{(k')} - X_{N_n}^{(k'+1)}
\le (\log N_n)^{-\gamma} X_{N_n}^{(k')}$.

\item To handle the case
when $y,z \in \cB_m \setminus \cB_{N_n} \subseteq
\cB_{N_{n+1}} \setminus \cB_{N_{n}}$,
it is sufficient to observe that $\xi(y)$ and $\xi(z)$
must take large values, because of \eqref{eq:lowinc}.
More precisely, on the event $\cC_{N_n}^c$,
cf. \eqref{eq:Cn}, both $\xi(y)$ and $\xi(z)$ must be larger than
$m^{d/\alpha}/(\log m)^{2\beta/\alpha} \ge
N_n^{d/\alpha}/(\log N_{n+1})^{2\beta/\alpha}$.

\item Consider finally the case when
exactly one of the points $y,z$ lies in $\cB_{N_{n}}$.
If $y \in \cB_{N_n}$ and $z \in \cB_m \setminus \cB_{N_n}$,
we have  $\xi(y) = X_{N_n}^{(k')}$ for some
$k' \le (\log m)^\beta$, as we
have already remarked, therefore
$0 < X_{N_n}^{(k')} - \xi(z) < (\log m)^{-\gamma} X_{N_n}^{(k')}$.
Viceversa, if $z \in \cB_{N_n}$
and $y \in \cB_m \setminus \cB_{N_n}$,
we may write
$0 < \xi(y) - X_{N_n}^{(k'')} < (\log m)^{-\gamma} \xi(y)$,
for some $k'' \le (\log m)^\beta$.
In either case, we can state that
there exists some point $x \in 
\cB_{N_{n+1}} \setminus \cB_{N_{n}}$ and some
$\bar k \le (\log N_{n+1})^\beta$ such that
$(1 - (\log N_n)^{-\gamma}) < \xi(x)/X_{N_n}^{(\bar k)} 
< (1 - (\log N_n)^{-\gamma})^{-1}$.
\end{enumerate}
These considerations lead us
directly to the following basic decomposition:
\begin{equation*}
	\tilde{\cV}_n^c \;\subseteq\;
	\cW_{n}^{(1)} \,\cup\, 
	\big( \cC_{N_n} \,\cup\, \cW_{n}^{(2)} \big) \,\cup\,
	\cW_{n}^{(3)} \,,
\end{equation*}
where the event $\cC_{N}$ has been introduced in \eqref{eq:Cn} and
we have set
\begin{equation*}
\begin{split}
	\cW_{n}^{(1)} & := 
	\union_{k' \le (\log N_{n+1})^\beta}
	\left\{ X_{N_n}^{(k'+1)}> \left( 1- \frac{1}{(\log N_n)^\gamma}
	\right) X_{N_n}^{(k')}\right\} \,, \\
	\cW_{n}^{(2)} & := 
	\union_{y , z \in \cB_{N_{n+1}}
	\setminus \cB_{N_n}, \; y \ne z}
	\Bigg\{\xi(y)
	\geq \frac{N_n^{d/\alpha}}{(\log N_{n+1})^{2\beta/\alpha}},
	\; \xi(z)
	\geq \frac{N_n^{d/\alpha}}{(\log N_{n+1})^{2\beta/\alpha}} 
	\Bigg\} \,, \\
	\cW_{n}^{(3)} & :=
	\union_{x \in \cB_{N_{n+1}} \setminus \cB_{N_n}, 
	\, \bar k\leq (\log N_{n+1})^\beta}
	\Bigg\{ 1- \frac{1}{(\log N_n)^\gamma}
	< \frac{\xi(x)}{X_{N_n}^{(\bar k)}}
	< \left( 1 - \frac{1}{(\log N_n)^\gamma} \right)^{-1} \Bigg\} \,.
\end{split}
\end{equation*}
Note that, by \eqref{eq:Cn}, $\sum_{n\in\N} \bbP(\cC_{N_n}) \le
\sum_{n\in\N} \frac{1}{N_n} \le
\sum_{n\in\N} e^{-n^r + 1} < \infty$.
By the Borel-Cantelli lemma, it suffices to show that
$\sum_{n \in \N} \bbP(\cW_n^{(i)}) < \infty$ for $i=1,2,3$
and it will follow that $\tilde \cV_n$ holds eventually $\bbP$-a.s.,
that is what we want to prove.

Let us consider $\cW_n^{(1)}$.
By \eqref{eq:Xk} we have
$\bbP \big( X_N^{(k+1)}\geq (1-\gep) X_N^{(k)} \big) \le c \, k \, \gep$
for some constant $c > 0$. Recalling that $N_n = e^{n^r}$,
for large $n$ we have
\begin{equation}\label{probaw}
\begin{split}
	\bbP(\cW_{n}^{(1)}) & \;\le\;
	\sum_{k=1}^{\lfloor (\log N_{n+1})^\beta \rfloor}
	\bbP \left( X_{N_n}^{(k+1)}> 
	\left( 1-\frac{1}{(\log N_n)^\gamma}
	\right) X_{N_n}^{(k)} \right) \\
	& \;\le\;
	c \, \frac{1}{(\log N_n)^\gamma}
	\,\sum_{k=1}^{\lfloor (\log N_{n+1})^\beta \rfloor} k 
	\;\le\; c' \, \frac{(\log N_{n+1})^{2\beta}}
	{(\log N_n)^\gamma} \;\le\; \frac{c''}{n^{r(\gamma - 2 \beta)}} \,,
\end{split}
\end{equation}
for suitable $c, c'' > 0$. It follows that
$\sum_{n\in\N} \bbP(\cW_n^{(1)}) < \infty$ provided
$r(\gamma - 2\beta) > 1$.

Next we consider $\cW_n^{(2)}$.
Observe that there exist constants $c, c' >0$ such that 
\begin{equation}\label{eq:cardi}
	|\cB_{N_{n+1}} \setminus \cB_{N_n}|
	\;\le\; c \, (N_{n+1} - N_n) \, (N_n)^{d-1} \;\le\;
	c' \frac{(N_n)^d}{n^{1-r}} \,,
\end{equation}
because $N_{n+1} - N_n = \lfloor e^{(n+1)^r} \rfloor - 
\lfloor e^{n^r} \rfloor = e^{n^r} r\, n^{r-1} (1+o(1))$ as $n \to \infty$.
Recalling that $\bbP(\xi(x) > t) \le t^{-\alpha}$
by \eqref{eq:pareto}, for a suitable $c''>0$ we can write
\begin{equation*}
\begin{split}
	\bbP(\cW_{n}^{(2)}) & \;\le\;
	\sum_{y, z \in \cB_{N_{n+1}}\setminus \cB_{N_n}, \, y \ne z} 
	\bbP \bigg( \xi(y)>  \frac{N_n^{d/\alpha}}
	{(\log N_{n+1})^{2\beta/\alpha}} \bigg)^2\\
	& \;\le\; (c')^2 \, \frac{(N_n)^{2d}}{n^{2(1-r)}} \,
	\frac{(\log N_{n+1})^{4\beta}}{(N_n)^{2d}} \;\le\;
	c'' \, \frac{n^{4\beta r}}{n^{2(1-r)}} \;=\;
	\frac{c''}{n^{2 - (4\beta + 2) r}} \,.
\end{split}
\end{equation*}
Therefore $\sum_{n\in\N} \bbP(\cW_{n}^{(2)})<\infty$
provided $2 - (4\beta + 2) r > 1$.

We finally focus on $\cW_n^{(3)}$.
Note that by \eqref{eq:pareto} for all $t > 1$
and $\gep < \frac 12$ we can write
\begin{equation}
	\bbP \bigg( (1-\gep) < \frac{\xi(x)}{t}
	< (1-\gep)^{-1} \bigg) \;=\;
	\int_{(1-\gep)t}^{(1-\gep)^{-1}t}
	\frac{\alpha}{s^{1+\alpha}} \, \dd s
	\;\le\; c \, \alpha \, \frac{\gep}{t^{\alpha}} \,,
\end{equation}
for some universal constant $c>0$. Note that
$\xi(x)$ is independent of $X_{N_n}^{(k)}$ if
$x \not \in \cB_{N_n}$.
If we are on the event $\cC_{N_n}^c$, cf. \eqref{eq:Cn},
$X_{N_n}^{(k)} \ge (N_n)^{d/\alpha}/(\log N_n)^{2\beta/\alpha}$
for $k \le (\log N_n)^\beta$, hence
\begin{equation*}
	\bbP \Bigg( 1- \frac{1}{(\log N_n)^\gamma}
	< \frac{\xi(x)}{X_{N_n}^{(k)}}
	< \left( 1 - \frac{1}{(\log N_n)^\gamma} \right)^{-1},
	\; \cC_{N_n}^c \Bigg) \;\le\;
	c \, \alpha \, \frac{1}{(\log N_n)^\gamma}
	\, \frac{(\log N_n)^\alpha}{(N_n)^d} \,.
\end{equation*}
Recalling \eqref{eq:Cn}, it follows that
\begin{equation*}
\begin{split}
	\bbP(\cW_{n}^{(3)}) & \;\le\; \bbP(\cC_{N_n}) +
	\bbP(\cW_{n}^{(3)}, \, \cC_{N_n}^c)  \;\le\; \frac{1}{N_n} +
	(\log N_{n+1})^\beta \, |\cB_{N_{n+1}}| \cdot 
	c \, \alpha \, \frac{(\log N_n)^{\alpha - \gamma}}{(N_n)^d} \\
	& \;\le\; c' \left( \frac{1}{e^{n^r}} +
	\frac{1}{n^{r(\gamma - \alpha - \beta)}} \right) \,,
\end{split}
\end{equation*}
for a suitable constant $c' >0$. If $r(\gamma - \alpha - \beta) >1$
we then have $\sum_{n\in\N} \bbP(\cW_{n}^{(3)})<\infty$.

The proof is completed observing that the three relations
we have found, namely
\begin{equation*}
	r(\gamma - 2 \beta) > 1\,, \qquad
	2 - (4 \beta + 2)r > 1\,, \qquad
	r(\gamma - \alpha - \beta) > 1\,,
\end{equation*}
can be satisfied at the same time. In fact, for any fixed $\beta$,
we can choose $r \in (0,1)$ small enough such that the second 
relation holds (e.g. $r := (4\beta + 3)^{-1}$)
and then choose $\gamma > 0$ large enough
so that the first and the third relations are satisfied
(e.g. $\gamma := 6 \beta + \alpha + 3$).

\medskip
\section{Order statistics for the modifed field}
\label{ap:modfield}

\smallskip

\subsection{Proof of Lemma~\ref{th:lemZ}}
\label{sec:lemZproof}

We are going to prove the following stronger result.

\begin{lemma}\label{lem:app_con_znkznun}
For all $k\ge 2$ and $\delta\in(0,1)$ one has
  \begin{gather}
    \label{eq:so1}
    \prob{\znk \ge (1-\delta) \znu} 
    \le (1 -(1-\delta)^{\alpha})^{k-1} \,.
  \end{gather}
\end{lemma}

\begin{proof}
We set
$L_{A} := \sup_{x\in A} \psi_N(x)$ (recall \eqref{eq:psi})
and $A^c := \cB_N \setminus A$ for short.
We also set $\phi_N(x) := (1-\frac{|x|}{N+1})$, so that
$\psi_N(x) = \phi_N(x) \xi(x)$.
Summing over the location of the set $A = \{Z_N^{(1)}, \ldots, Z_N^{(k-1)}\}$,
so that $Z_N^{(k)} = L_{A^c}$, we can write
\begin{align}
	\nonumber
	\bbP( Z_N^{(k)} & \ge (1-\delta) Z_N^{(1)} )
	= \sum_{A \subseteq \cB_N,\, |A|=k-1}
	\bbP\left( \{Z_N^{(1)}, \ldots, Z_N^{(k-1)}\} = A, \;
	L_{A^c} \ge (1-\delta) Z_N^{(1)} \right) \\
	\label{eq:tre}
	& = \sum_{A \subseteq \cB_N,\, |A|=k-1}
	\bbP\left( L_{A^c} < \psi_N(x) \le (1-\delta)^{-1} L_{A^c} \,, \;
	\forall x \in A \right) \\
	\nonumber
	& = \sum_{A \subseteq \cB_N,\, |A|=k-1}
	\bbP\left( (1-\delta)^\alpha \left(\frac{\phi_N(x)}{L_{A^c}}\right)^\alpha
	\le \frac{1}{\xi(x)^\alpha} < \left(\frac{\phi_N(x)}{L_{A^c}}\right)^\alpha \,, \;
	\forall x \in A \right) \,.
\end{align}
It follows from \eqref{eq:pareto} that
the variable $1/\xi(x)^\alpha$ 
is uniformly distributed on the interval $(0,1)$, that is,
its distribution function equals $J(x) := (x \wedge 1) \ind_{(0,\infty)}(x)$,
hence
\begin{equation*}
	\bbP\left( (1-\delta)^\alpha t^\alpha
	\le \frac{1}{\xi(x)^\alpha} < t^\alpha \right) = J(t^\alpha) -
	J((1-\delta)^\alpha t^\alpha) \,.
\end{equation*}
One checks easily that $J((1-\delta)^\alpha t^\alpha) \ge (1-\delta)^\alpha J(t^\alpha)$
for all $\delta \in (0,1)$ and $t \ge 0$
(the inequality is strict for $t > 1$), therefore
\begin{equation} \label{eq:treineq}
	\bbP( Z_N^{(k)} \ge (1-\delta) Z_N^{(1)} ) \le
	(1 - (1-\delta)^\alpha)^{k-1} \sum_{A \subseteq \cB_N,\, |A|=k-1}
	\bbE \Bigg[ \prod_{x\in A} J \left(
	\frac{\phi_N(x)^\alpha}{(L_{A^c})^\alpha} \right) \Bigg] \,.
\end{equation}
Setting $\delta = 1$ in \eqref{eq:tre} we see that the
sum in the right hand side of the last equation equals one,
and the proof is completed.
\end{proof}

\begin{remark}\rm
One can refine the proof of Lemma~\ref{lem:app_con_znkznun} to show that
\begin{equation*}
	\prob{\znk \ge (1-\delta) \znu} \ge
	\big( 1 - C_k \, e^{-c_k N^d} \big) \, (1 -(1-\delta)^{\alpha})^{k-1} \,,
\end{equation*}
for suitable constants $c_k, C_k \in (0,\infty)$ and for large $N$. In fact,
restricting the expectations in \eqref{eq:tre} to the event
$\{Z_N^{(k)} > 1\}$, one has $\phi_N(x)/L_{A^c} \le 1$ and therefore \eqref{eq:treineq}
becomes
\begin{equation*}
	\bbP \big( Z_N^{(k)} \ge (1-\delta) Z_N^{(1)},\; Z_N^{(k)} > 1 \big) = 
	(1 - (1-\delta)^\alpha)^{k-1} \, \bbP( Z_N^{(k)} > 1) \,.
\end{equation*}
It then remains to check that $\bbP( Z_N^{(k)} > 1) \le C_k \exp(- c_k N^d)$,
which can be easily done by direct computation.
\end{remark}

\smallskip
\subsection{Proof of Lemma~\ref{cori}}
\label{sec:cori}

As already remarked, only the first inequality
in \eqref{eq:lbZ1} needs to be proved, because $Z_N^{(2)} \le Z_N^{(1)} \le X_N^{(1)}$
(recall \eqref{eq:ubX1}). We start with an auxiliary lemma.

\begin{lemma}
There exist constants $c_1, c_2$ such that for all $N \in \N$ and $t \ge 0$
\begin{gather} \label{eq:Z1t}
	\bbP (Z_N^{(2)} \le N^{d/\ga}\, t) \le c_1 \,
        e^{-\frac{c_2}{t^\ga}}  \,.
\end{gather}
\end{lemma}

\begin{proof}
Setting $O_x := \sup_{x \in \cB_N \setminus \{x\}} \psi_N(x)$ for short,
we can write
\begin{equation*}
\begin{split}
	\bbP (Z_N^{(2)} \le N^{d/\ga}\, t) & = \sum_{x \in \cB_N}
	\bbP (O_x \le N^{d/\ga}\, t,\, \xi(x) > O_x) \\
	& = \sum_{x \in \cB_N} \bbP \left( \frac{1}{O_x^\alpha} \ge \frac{1}{N^{d}\, t^\ga},
	\, \frac{1}{\xi(x)^\alpha} < \frac{1}{O_x^\alpha} \right) 
	\le \sum_{x \in \cB_N}
	\bbE \left( \frac{1}{O_x^\alpha} \,
	\ind_{\left\{\frac{1}{O_x^\alpha} \ge \frac{1}{N^{d}\, t^\ga}\right\}} \right) \,,
\end{split}
\end{equation*}
because $1 / \xi(x)^\alpha$ is uniformly distributed on the interval $(0,1)$,
as it follows from \eqref{eq:pareto}. We then apply the basic formula
$\bbE(Z \, \ind_{\{Z \ge a\}}) = a \bbP(Z \ge a) + \int_a^\infty \bbP(Z \ge s) \, \dd s$,
getting
\begin{equation*}
	\bbP (Z_N^{(2)} \le N^{d/\ga}\, t)  \le
	\frac{1}{N^{d}} \sum_{x \in \cB_N} \left\{ \frac{1}{t^\ga} \,
	\bbP \left( \frac{1}{O_x^\alpha} \ge \frac{1}{N^{d}\, t^\ga} \right)
	+ \int_{t^{-\alpha}}^\infty 
	\bbP \left( \frac{1}{O_x^\alpha} \ge \frac{u}{N^{d}} \right)
	\, \dd u \right\} \,.
\end{equation*}
We now claim that there exists $c > 0$ such that for all $N \in \N$,
$x \in \cB_N$ and $u > 0$
\begin{equation} \label{eq:totoprove}
	\bbP \left( \frac{1}{O_x^\alpha} \ge \frac{u}{N^{d}} \right) \le
	e^{-c \, u} \,.
\end{equation}
Since $|\cB_N| \le C N^d$ for some constants $C$, we get
\begin{equation*}
	\bbP (Z_N^{(2)} \le N^{d/\ga}\, t)  \le \frac{C}{t^\alpha}
	\, e^{-c \, t^{-\alpha}} \,+\, C \int_{t^{-\alpha}}^\infty 
	e^{-c \, u} \, \dd u \le e^{-c \, t^{-\alpha}}
	\left( \frac{C}{t^\alpha} + \frac{C}{c} \right) \,.
\end{equation*}
Since the function $t \mapsto t^{-\alpha} \, e^{-\frac 12 c \, t^{-\alpha}}$
is bounded on $\R^+$, it follows that \eqref{eq:Z1t}
holds true with $c_2 := \frac 12 c$ and for $c_1$ large enough.

It remains to prove \eqref{eq:totoprove}, for which we can write
\begin{equation*}
\begin{split}
	\bbP \left( \frac{1}{O_x^\alpha} \ge \frac{u}{N^{d}} \right)
	& = \prod_{z \in \cB_N \setminus \{x\}}
	\bbP \left( \frac{1}{\xi(z)^\alpha} \ge 
	\frac{\big(1-\tfrac{|z|}{N+1}\big)^\alpha}{N^d} u \right) \\
	& \le \exp \Bigg( \frac{u}{N^d} \sum_{z \in \cB_N \setminus \{x\}}
	\big( 1-\tfrac{|z|}{N+1} \big)^\alpha \Bigg) \,,
\end{split}
\end{equation*}
because $\bbP(1/\xi(z)^\alpha \ge a) = 1-a \le e^{-a}$ for $a \in [0,1]$
(recall \eqref{eq:pareto}). By a Riemann sum approximation, as $N\to\infty$ one has
\begin{equation*}
	\frac{1}{N^d} \sum_{z \in \cB_N \setminus \{x\}}
	\big( 1-\tfrac{|z|}{N+1} \big)^\alpha \;\longrightarrow\;
	\int_{|y| \le 1} (1-|y|)^\alpha \, \dd y \;\in\; (0,\infty) \,,
\end{equation*}
from which it follows that \eqref{eq:totoprove} holds true for some $c > 0$.
\end{proof}

\begin{proof}[Proof of Lemma~\ref{cori}]
Thanks to the inequality \eqref{eq:Z1t}, the proof is
identical to that of the
lower bound in \eqref{eq:ubX1}, cf. Appendix~\ref{sec:ubX1proof}. More precisely,
one first shows, through a standard Borel-Cantelli argument, that
the first inequality in \eqref{eq:lbZ1} (with $\epsilon$ replaces by $\epsilon/2$, say)
holds along the subsequence $N_k := 2^k$; the extension to all
values of $N$ then follows easily, because $Z_N^{(2)}$ is increasing in $N$.
We omit the details for conciseness.
\end{proof}

%
%

\subsection{Further results}
\label{sec:gap12?}

It may be useful to observe that if
$z_{N+1}^{(1)} \ne z_{N}^{(1)}$ then
\begin{equation} \label{eq:claim}
        |z_{N+1}^{(1)}| > |z_{N}^{(1)}| \qquad \text{and} \qquad
        \xi(z_{N+1}^{(1)}) > \xi(z_{N}^{(1)}) \,.
\end{equation}
In fact, when $z_{N+1}^{(1)} \ne z_{N}^{(1)}$ we have
by definition
\begin{equation} \label{eq:N+1}
        Z_N^{(1)} = \psi_N(z_N^{(1)}) > \psi_N (z_{N+1}^{(1)}) \,, \qquad
        \psi_{N+1}(z_N^{(1)}) < \psi_{N+1} (z_{N+1}^{(1)}) = Z_{N+1}^{(1)} \,,
\end{equation}
from which we obtain, recalling the definition \eqref{eq:psi} of $\psi_N$,
\begin{equation*}
        \frac{|z_N^{(1)}| \, \xi(z_N^{(1)})}{(N+1)(N+2)} = \psi_{N+1}(z_N^{(1)})
        - \psi_{N}(z_N^{(1)}) < \psi_{N+1}(z_{N+1}^{(1)})
        - \psi_{N}(z_{N+1}^{(1)})  =
        \frac{|z_{N+1}^{(1)}| \, \xi(z_{N+1}^{(1)})}{(N+1)(N+2)} \,,
\end{equation*}
hence $|z_N^{(1)}| \, \xi(z_N^{(1)}) < |z_{N+1}^{(1)}| \, \xi(z_{N+1}^{(1)})$.
This shows that at least one of the two inequalities
in \eqref{eq:claim} must hold. Two cases remain that need
to be excluded:
\begin{itemize}
\item if $|z_{N+1}^{(1)}| \le |z_{N}^{(1)}|$ and
$\xi(z_{N+1}^{(1)}) > \xi(z_{N}^{(1)})$, then
\begin{equation*}
        \psi_N(z_{N+1}^{(1)}) = \left( 1- \frac{|z_{N+1}^{(1)}|}{N+1} \right)
        \xi(z_{N+1}^{(1)}) > \left( 1- \frac{|z_{N}^{(1)}|}{N+1} \right)
        \xi(z_{N}^{(1)}) = \psi_N(z_N^{(1)}) = Z_N^{(1)}\,,
\end{equation*}
which is absurd, because $Z_N^{(1)}$ is by definition the maximum of $\psi_N$;

\item analogously, if $|z_{N+1}^{(1)}| > |z_{N}^{(1)}|$ and
$\xi(z_{N+1}^{(1)}) \le \xi(z_{N}^{(1)})$, then
\begin{equation*}
        Z_{N+1}^{(1)} =
        \psi_{N+1}(z_{N+1}^{(1)}) = \left( 1- \frac{|z_{N+1}^{(1)}|}{N+2} \right)
        \xi(z_{N+1}^{(1)}) < \left( 1- \frac{|z_{N}^{(1)}|}{N+2} \right)
        \xi(z_{N}^{(1)}) = \psi_{N+1}(z_N^{(1)}) \,,
\end{equation*}
which is again absurd, because $Z_{N+1}^{(1)}$ is by definition
the maximum of $\psi_{N+1}$.
\end{itemize}

\smallskip

Next we show that a statement analogous to \eqref{eq:gap13}
for the gap $Z_N^{(1)} - Z_N^{(2)}$ \emph{does not hold}.
Let us fix any $\bar N$ for which
$z_{\bar N}^{(1)} \ne z_{\bar N+1}^{(1)}$ (note that there are
almost surely infinitely
many such values of $\bar N$, otherwise $Z_N^{(1)} = \psi_N(z_N^{(1)})$
would be eventually constant). We set
$x := z_{\bar N}^{(1)}$ and $y := z_{\bar N+1}^{(1)}$ for short.
Then $Z_{\bar N}^{(1)} = \psi_N(x)$ and $Z_{\bar N}^{(2)} \ge \psi_N(y)$, hence,
recalling \eqref{eq:psi},
\begin{equation*}
\begin{split}
	Z_{\bar N}^{(1)} - Z_{\bar N}^{(2)} & \le \psi_{\bar N}(x) - \psi_{\bar N}(y)
	= \frac{1}{{\bar N}+1} \Big( ({\bar N}+1)\big( \xi(x) - \xi(y) \big)
	+ |y|\xi(y) - |x|\xi(x) \Big) \\
	& = \frac{1}{{\bar N}+1} \Big( ({\bar N}+2)\big( \xi(x) - \xi(y) \big)
	+ |y|\xi(y) - |x|\xi(x) \Big) \,+\, \frac{\xi(y) - \xi(x)}{{\bar N}+1} \\
	& = \frac{{\bar N}+2}{{\bar N}+1} \Big( \psi_{\bar N +1}(x) 
	- \psi_{\bar N + 1}(y) \Big) \,+\, \frac{\xi(y) - \xi(x)}{{\bar N}+1} \,.
\end{split}
\end{equation*}
By construction $y = z_{\bar N+1}^{(1)}$ and $y \ne x$, therefore
$\psi_{\bar N +1}(y) = Z_{\bar N+1}^{(1)} > \psi_{\bar N +1}(x)$.
Recalling \eqref{eq:ubX1}, we infer that eventually $\bbP$-a.s.
\begin{equation} \label{eq:roughest}
	Z_{\bar N}^{(1)} - Z_{\bar N}^{(2)} \le \frac{\xi(z_{\bar N+1}^{(1)}) 
	- \xi(z_{\bar N}^{(1)})}{{\bar N}+1} \le
	\frac{X_{\bar N+1}^{(1)}}{\bar N + 1} \le
	{\bar N}^{d/\alpha - 1} \, (\log \bar N)^{1/\alpha + \epsilon} \,.
\end{equation}
We stress that this bound differs from the one in \eqref{eq:gap13}
almost by a factor $N^{-1}$.
It turns out that the bound \eqref{eq:roughest} is quite sharp
(up to logarithmic corrections): in fact, by the first bound
in \eqref{eq:gap123}, \eqref{eq:lbZ1} and a Borel-Cantelli argument,
it follows that for every $\epsilon > 0$, eventually $\bbP$-almost surely,
\begin{equation} \label{eq:roughest2}
	Z_N^{(1)} - Z_N^{(2)} \ge \frac{Z_N^{(1)}}{N (\log N)^{1+\epsilon/2}}
	\ge \frac{N^{d/\alpha - 1}}{(\log N)^{1 + \epsilon}} \,.
\end{equation}
This implies in particular that
$N(Z_N^{(1)} - Z_N^{(2)}) \to +\infty$, $\bbP$-almost surely.


\medskip
\section{Proof of \eqref{eq:wz112} in Remark \ref{re:zn2}}
\label{ap:zn2}

We want to prove, for $d=1$, that 
\begin{equation} \label{eq:toz2}
	\bbP\left(\rw_{N,\xi} = z_{N,\xi}^{(2)} 
	\text{ for infinitely many $N$} \right) \,=\, 1 \,.
\end{equation}
To simplify notation, we only consider the case
$\alpha > 1$ and we set $m_\alpha=\bbE(\xi_1)=\alpha/(\alpha-1)$,
cf. \eqref{eq:pareto}. Recalling \eqref{eq:rw0}, we set $\kappa = \p(S_1=0)$ 
and $\hat{\kappa}=\p(S_1 = 1)$. For the sake of simplicity,
we also assume that $\log (\hat{\kappa}/\kappa) >-m_\alpha$.
The cases $\log (\hat{\kappa}/\kappa)< -m_\alpha$ and
$\log (\hat{\kappa}/\kappa)=-m_\alpha$ are controlled with analogous arguments.

For $\alpha,\eta > 0$ and for $n\in \N$ and 
$\gep>0$ we define the event $B_{\epsilon,\eta,n} \subseteq \Omega_\xi$ by
\begin{equation}\label{eq:defB}
\begin{split}
 	B_{\epsilon,\eta,n} \,:=\, \big\{ & \exists ! x \in [n, (1+\gep)n]: \,
	\xi(x) \in (1, 1+\gep) n^{1/\alpha} \,, \\
 	& \exists ! y \in [3n, (1+\gep) 3n]: \,
	\xi(y) \in (1, 1+\gep) \tfrac53n^{1/\alpha} \,, \\
 	& \forall z \in [-7n, (1+\epsilon)7n] \setminus\{x,y\}: \,
	\xi(z) < \tfrac 12 n^{1/\alpha}\,,\\
	&\qquad \qquad \quad \textstyle\sum_{i=x+1}^{y-1} \xi(i) >(m_\alpha-\eta) 
	(y-x) \big\}  \,,
\end{split}
\end{equation}
where we set $[a,b] := [a,b] \cap \Z$ for short.
By direct computation, one checks
easily that $\lim_{n\to\infty}\bbP(B_{\epsilon,\eta, n}) > 0$
for all fixed $\epsilon, \eta > 0$.
In what follows, we denote by $x, y$ the (random) points
appearing in the definition of $B_{\epsilon,\eta, n}$. 

Recalling \eqref{eq:psi}, for all $N \in \N$ we have
\begin{equation} \label{eq:tresplit}
\begin{split}
	\left( 1 - \tfrac{n}{N}(1+\gep) \right) n^{1/\alpha}
	\,<\, \psi_N(x) \,<\, \left( 1 - \tfrac{n}{N} \right) (1+\gep)  n^{1/\alpha}, \\
	\left( 1 - \tfrac{3n}{N} (1+\gep) \right) \tfrac53 n^{1/\alpha}
	\,<\, \psi_N(y) \,<\, \left( 1 - \tfrac{3n}{N} \right) (1+\gep)
	\tfrac53 n^{1/\alpha} \,.
\end{split}
\end{equation}
It follows that for all $N \in [\tfrac{11}{2}n, \tfrac{13}{2}n]$ we have
$\psi_N(x) > (1 - \frac{1+\gep}{11/2}) n^{1/\alpha} = (\frac{9}{11}+O(\gep)) n^{1/\alpha}$,
$\psi_N(y) > (1 - \frac{3(1+\gep)}{11/2}) \frac{5}{3} n^{1/\alpha}
=(\frac{25}{33}+O(\gep))  n^{1/\alpha}$,
while
$\psi_N(z) < \frac 12 n^{1/\alpha}$ for all
$z \in [-N, N] \setminus \{x,y\}$.
Therefore, by choosing $\gep$ small enough, we can state that on the
event $B_{\epsilon,\eta, n}$ and 
for all $n \in \N$, $N \in [\tfrac{11}{2}n,\tfrac{13}{2} n]$ it comes
\begin{equation}\label{eq:cond1}
	\{z_N^{(1)}, z_N^{(2)}\} = \{x,y\}.
\end{equation}

For $N = \frac{11}{2} n$
we have $\psi_N(x) = (\frac{9}{11} + O(\epsilon)) n^{1/\alpha}$ and
$\psi_N(y) = (\frac{25}{33} + O(\epsilon)) n^{1/\alpha}$, uniformly in $n$;
on the other hand, for $N = \frac{13}{2} n$
we have $\psi_N(x) = (\frac{11}{13} + O(\epsilon)) n^{1/\alpha}$ and
$\psi_N(y) = (\frac{35}{39} + O(\epsilon)) n^{1/\alpha}$, always uniformly in $n$.
It follows that, if $\epsilon > 0$ is chosen small enough
we have for all $n\in\N$,
\begin{equation}\label{eq:cond2}
\psi_{\frac{11}{2} n}(y) - \psi_{\frac{11}{2} n}(x) < 0\quad \text{but}\quad
\psi_{\frac{13}{2} n}(y) - \psi_{\frac{13}{2} n}(x) > 0.
\end{equation}
At this stage, we pick $\gep_0>0$ such that \eqref{eq:cond1} and \eqref{eq:cond2} are satisfied. Next observe that
\begin{equation} \label{eq:display}
	(N+1)(\psi_N(y) - \psi_N(x)) = x\xi(x) - y\xi(y)
	+ (N+1)(\xi(y) - \xi(x))
\end{equation}
is increasing in $N$, because by construction $(\xi(y) - \xi(x)) > 0$.
It follows that there is $N^*_n \in (\frac{11}{2}n, \frac{13}{2}n)$ such that:
\begin{itemize}
\item for $\frac{11}{2}n < N \le N_n^*$ we have $(\psi_N(y) - \psi_N(x)) < 0$, hence
$x = z_N^{(1)}$ and $y = z_N^{(2)}$;
\item for $N_n^* < N < \frac{13}{2} n$ we have $(\psi_N(y) - \psi_N(x)) > 0$, hence
$x = z_N^{(2)}$ and $y = z_N^{(1)}$.
\end{itemize}
By \eqref{eq:display} $(N+1)(\psi_N(y) - \psi_N(x))$ increases
by $(\xi(y) - \xi(x))$ when $N$ increases by $1$.
Since $(N_n^*+1)(\psi_{N_n^*}(y) - \psi_{N_n^*}(x)) < 0$
and $((N_n^*+1)+1)(\psi_{N_n^*+1}(y) - \psi_{N_n^*+1}(x)) > 0$,
it then follows that
$(N_n^*+1)(\psi_{N_n^*}(y) - \psi_{N_n^*}(x)) > -(\xi(y)-\xi(x))$, that is
\begin{equation} \label{eq:bound1212}
\begin{split}
	(N_n^*+1) (Z_{N_n^*}^{(1)} - Z_{N_n^*}^{(2)}) 
	& = ({N_n^*}+1)(\psi_{N_n^*}(x) - \psi_{N_n^*}(y)) \le
	(\xi(y)-\xi(x)) 	\\
	& \le \left( \frac{2}{3} + \frac{5}{3}\epsilon_0 \right) n^{1/\alpha}=:c_0\, n^{1/\alpha} \,,
\end{split}
\end{equation}
by the definition of the event $B_{\epsilon_0,\eta, n}$.

Consider now the contributions of the two $N$-steps random walk trajectories
$\cS^{(N, x)}$ and $\cS^{(N, y)}$ that reach respectively
$x$ and $y$ in the minimal number of steps and stick there
until time $N$, i.e.,
\begin{align*}
	\bP_{N,\xi}(\cS^{(N,x)}) 
	&=\kappa^N e^{\sum_{i=1}^{x-1} \xi(i) + (N+1) \psi_N(x)+x\log(\tfrac{\hat{\kappa}}{\kappa})},\\
\bP_{N,\xi}(\cS^{(N,y)})
	& =\kappa^N  e^{\sum_{i=1}^{y-1} \xi(i) + (N+1) \psi_N(y)+y\log(\tfrac{\hat{\kappa}}{\kappa})}, 
\end{align*}
so that 
\begin{equation}\label{eq:compaZ}
 \frac{\bP_{N,\xi}(\cS^{(N,y)})}{\bP_{{N},\xi}(\cS^{(N,x)})}
	 =
	\, e^{\sum_{i=x}^{y-1}\xi(i) - (N+1)(\psi_N(x)-\psi_N(y))+(y-x)\log(\tfrac{\hat{\kappa}}{\kappa})}. \\
\end{equation}
We apply this relation for $N=N_n^*$ on the event $\cB_{\gep_0,\eta_0,n}$,
with $2\,\eta_0 := m_\alpha+\log(\hat{\kappa}/\kappa)$ (which is
strictly positive, by
our initial assumtion). Then $x=z_{N_n^*}^{(1)}$, $y=z_{N_n^*}^{(2)}$ and \eqref{eq:compaZ} becomes
\begin{equation}
\begin{split}
\frac{\bP_{N_n^*,\xi}\big(\cS^{(N_n^*,z_{N_n^*}^{(2)})}\big)}{\bP_{N_n^*,\xi}\big(\cS^{(N_n^*,z_{N_n^*}^{(1)})}\big)}&\geq 
e^{\sum_{i=x+1}^{y-1} \xi(i)-(N_n^*+1)(Z_{N_n^*}^{(1)}-Z_{N_n^*}^{(2)})+(y-x)\log(\tfrac{\hat{\kappa}}{\kappa})},\\
&\geq e^{(m_\alpha-\eta+\log(\tfrac{\hat{\kappa}}{\kappa})) (y-x)-c_0 n^{1/\alpha}} \geq e^{\eta n-c_0 n^{1/\alpha}},
\end{split}
\end{equation}
where we have used \eqref{eq:bound1212}, the last condition in \eqref{eq:defB} and the fact that $y-x\geq n$, again by \eqref{eq:defB}.
Since $\alpha > 1$ by assumption, we have shown that on the event $\cB_{\gep_0,\eta_0,n}$
$$
	\bP_{{N_n^*},\xi}(\cS^{(N_n^*, z_{N_n^*}^{(2)})}) \gg 
	\bP_{{N_n^*},\xi}(\cS^{(N_n^*, z_{N_n^*}^{(1)})}) \,.
$$
If $\bP_{{N},\xi}(\cS^{(N, z_{N}^{(2)})}) +
\bP_{{N},\xi}(\cS^{(N, z_{N}^{(1)})}) > \tfrac{3}{4}$, this shows that
$\rw_{N,\xi} = z_N^{(2)}$. To sum up,
there exists $n_0\in \N$ such that for $n\geq n_0$
\begin{align*}
	B_{\epsilon_0,\eta_0 , n} 
	\subseteq \{ \exists N \in (\tfrac{11}{2}n, &\tfrac{13}{2}n):\
	\rw_{N,\xi} = z_N^{(2)}\}\\
 &\cup
	\big\{  \exists N \in (\tfrac{11}{2}n, \tfrac{13}{2}n)\colon \bP_{{N},\xi}(\cS^{(N, z_{N}^{(2)})}) +
	\bP_{{N},\xi}(\cS^{(N, z_{N}^{(1)})}) \le \tfrac{3}{4} \big\} \,.
\end{align*}
Recalling that
$\bP_{{N},\xi}(\cS^{(N, z_{N}^{(2)})}) +
\bP_{{N},\xi}(\cS^{(N, z_{N}^{(1)})}) \to 1$ as $N \to \infty$,
$\bbP(\dd\xi)$-almost surely when $d=1$,
cf. Remark~\ref{rem:1d}, it follows that almost surely
\begin{equation*}
	\limsup_{n\to\infty}B_{\epsilon_0,\eta_0, n} :=
	\{ B_{\epsilon_0,\eta_0, n} \text{ for infinitely many } n\} \subseteq
	\{ \rw_{N,\xi} = z_N^{(2)} \text{ for infinitely many } N\} \,.
\end{equation*}

Finally, note that $\bbP(\limsup_{n\to\infty}B_{\epsilon_0,\eta_0, n}) \ge \lim_{n\to\infty}
\bbP(B_{\epsilon_0,\eta_0, n}) > 0$, and it is not difficult to realize that indeed
$\bbP(\limsup_{n\to\infty}B_{\epsilon_0,\eta_0, n}) = 1$, because when $m \gg n$
the event $B_{\epsilon_0,\eta_0, m}$ is asymptotically independent of 
$B_{\epsilon_0,\eta_0, n}$. This completes the proof.

\end{document}